\documentclass[11pt]{article}
\usepackage[utf8]{inputenc}
\usepackage{amsthm,amsmath,amsfonts}
\usepackage{fullpage}
\usepackage[utf8]{inputenc}
\usepackage{bbm}
\usepackage{hyperref}
\usepackage{graphicx}
\usepackage{verbatim}
\newcommand{\bE}{\mathbb{E}}

\newcommand{\E}{\bE}      % Expectation

\newcommand{\bone}{\mathbbm{1}}

\DeclareMathOperator*{\argmin}{arg\,min}

\DeclareMathOperator{\poly}{poly}

\usepackage{xcolor}

\newcommand{\ignore}[1]{}

\newtheorem{theorem}{Theorem}
\newtheorem{corollary}[theorem]{Corollary}
\newtheorem{lemma}[theorem]{Lemma}

\theoremstyle{definition}
\newtheorem{defn}[theorem]{Definition}
\theoremstyle{remark}
\newtheorem{remark}[theorem]{Remark}
\newtheorem{example}[theorem]{Example}
\newtheorem{question}[theorem]{Question}

\title{Reconstruction on Trees and Low-Degree Polynomials}
\author{Frederic Koehler\thanks{Department of Computer Science, Stanford University. Supported in part by E. Mossel’s Vannevar Bush Faculty Fellowship ONR-N00014-20-1-2826, NSF award CCF-1704417, NSF Award IIS-1908774, and N. Anari's Sloan Research Fellowship. Part of this work was completed while participating in the program \emph{Computational Complexity of Statistical Inference} at the Simons Institute for the Theory of Computing.} \and Elchanan Mossel\thanks{Department of Mathematics and IDSS, Massachusetts Institute of Technology. Supported by Simons-NSF collaboration on deep learning NSF DMS-2031883, by
Vannevar Bush Faculty Fellowship award ONR-N00014-20-1-2826 and by a Simons Investigator Award in Mathematics (622132)}}
\date{}

\begin{document}

\maketitle
\begin{abstract}
The study of Markov processes and broadcasting on trees has deep connections to a variety of areas including statistical physics, graphical models, phylogenetic reconstruction, Markov Chain Monte Carlo, and community detection in random graphs.
Notably, the celebrated Belief Propagation (BP) algorithm achieves Bayes-optimal performance for the reconstruction problem of predicting the value of the Markov process at the root of the tree from its values at the leaves.

Recently, the analysis of low-degree polynomials has emerged
as a valuable tool for predicting computational-to-statistical gaps. In this work, we investigate the performance of low-degree polynomials for the reconstruction problem on trees. Perhaps surprisingly, we show that there are simple tree models with $N$ leaves and bounded arity where (1) nontrivial reconstruction of the root value is possible with a simple polynomial time algorithm and with robustness to noise, but not with any polynomial of degree $N^{c}$ for $c > 0$ a constant depending only on the arity, and (2) when the tree is unknown and given multiple samples with correlated root assignments, nontrivial reconstruction of the root value is possible with a simple 
Statistical Query algorithm but not with any polynomial of degree $N^c$. These results clarify some of the limitations of 
low-degree polynomials vs. polynomial time algorithms
for Bayesian estimation problems. They also complement recent work of Moitra, Mossel, and Sandon who studied the circuit complexity of Belief Propagation.  As a consequence of our main result, we are able to prove a result of independent interest regarding the performance of RBF kernel ridge regression for learning to predict the root coloration: for some $c' > 0$ depending only on the arity, $\exp(N^{c'})$ many samples are needed for the kernel regression to obtain nontrivial correlation with the true regression function (BP). We pose related open questions about low-degree polynomials and the Kesten-Stigum threshold. 
\end{abstract}

\section{Introduction}
Understanding the computational complexity of random instances has been the goal of an extensive line of research spanning multiple decades and different research areas such as cryptography, high-dimensional statistics, complexity theory, and statistical physics.
In particular, this includes work on satisfiability and refutation of random constraint satisfaction problems  and on computational-to-statistical gaps. 
In much of this work, evidence for  computational hardness is indirect because there are well-known barriers to proving hardness from classical worst-case assumptions such as NP-hardness \cite{applebaum2008basing,bogdanov2006worst,feigenbaum1993random,akavia2006basing}.
%Moreover, the (conjectured) computational difficulty often corresponds to a phase transition in the parameters. 

Recently, low-degree polynomials have emerged as a powerful tool for predicting computational-to-statistical gaps. Computational-to-statistical gaps are situations where it is impossible for polynomial time algorithms to estimate a desired quantity of interest from the data, even though computationally inefficient (``information-theoretic'') algorithms can succeed at the same task.
Heuristics based on low-degree polynomials have especially been used in the context of Bayesian estimation and testing problems and partially motivated by connections with (lower bounds for) the powerful Sum-of-Squares proof system. 
More specifically, a recent line of work (e.g. \cite{hopkins2017efficient,hopkins2018statistical,kunisky2019notes, bandeira2019computational,gamarnik2020low,mao2021optimal,holmgren2020counterexamples,bresler2021algorithmic,wein2020optimal}) showed that a suitable ``low-degree heuristic'' can be used to predict computational-statistical gaps for a variety of problems such as recovery in the multicommunity stochastic block model, sparse PCA, tensor PCA, the planted clique problem, certification in the zero-temperature Sherrington-Kirkpatrick model, the planted sparse vector problem, and for finding solutions in random $k$-SAT problems. %  \todo{many other references} and
Furthermore, it was observed that the predictions from this method generally agree with those conjectured using other techniques (for example, statistical physics heuristics based on studying BP/AMP fixed points, see e.g. \cite{decelle2011asymptotic,deshpande2015finding,mohanty2021}). Some of the merits of the low-degree polynomial framework include that it is relatively easy to use (e.g.\ compared to proving SOS lower bounds), and that low degree polynomials capture the power of the ``local algorithms'' framework used in e.g.\ \cite{gamarnik2014limits,chen2019suboptimality} as well as algorithms which incorporate global information, such as spectral methods or a constant number of iterations of Approximate Message Passing \cite{wein2020optimal}.

%In this work we rigorously prove computational complexity lower bounds for average case inference for the reconstruction problem on trees. 
In this work, we investigate the power of low-degree polynomials for the (average case) reconstruction problem on trees. We define the model and results in the next sections, but first give an informal summary.
The goal for reconstruction on trees is to estimate the value of the Markov process at the root given its value at the leaves (in the limit where the depth of the tree goes to infinity), and
two key parameters of the model are the arity of the tree $d$ and the magnitude of the second eigenvalue $\lambda_2$ of the broadcast chain. Importantly, when $d |\lambda_2|^2 > 1$ it is known \cite{kesten1966additional} that nontrivial reconstruction of the root is possible just from knowing the counts of the leaves of different types, whereas when $d |\lambda_2|^2 < 1$ such count statistics have no mutual information with the root (but more complex statistics of the leaves may)~\cite{mossel2003information}. This threshold $d |\lambda_2|^2 = 1$ is known as the \emph{Kesten-Stigum threshold} \cite{kesten1966additional} and it plays a fundamental role in other problems, such as algorithmic recovery in the stochastic block model \cite{abbe2017community} and phylogenetic reconstruction \cite{daskalakis2006optimal}. Count statistics can be viewed as degree 1 polynomials of the leaves, which begs the question of what information more general polynomials can extract from the leaves. 

In this paper, we answer this question in the limit case $\lambda_2 = 0$. Perhaps surprisingly, we find that the Kesten-Stigum threshold remains tight in the sense that even polynomials of degree $N^c$ for a small $c > 0$ are not able to correlate with the root label (Theorem~\ref{thm:poly-fail-intro}), whereas computationally efficient reconstruction is generally possible as long as $d$ is a sufficiently large constant (Theorem~\ref{thm:noise-robust-reconstruction-intro}) and even when a constant fraction of leaves are replaced by noise. Building on the polynomial lower bound, we prove superpolynomiallly (in fact, subexponentially) many samples are needed for Gaussian Kernel Ridge Regression (KRR) to (weakly) learn to regression function which predicts the root from the leaf colorations (Theorem~\ref{thm:kernel}). This gives a simple and natural model where KRR provably fails that is outside the reach of existing lower bounds such as \cite{kamath2020approximate}.

We also consider an analogous question where the tree is unknown, and the algorithm has access to $m$ i.i.d. samples of the Markov process
 where the root is biased towards an unknown label $Y^*$. In this setting, polynomials of degree $N^c$ again fail to correlate with the label $Y^*$, but we show that a simple algorithm, straightforwardly implementable in the Statistical Query (SQ) model \cite{kearns1998efficient}, can recover $Y^*$ in polynomial time (Theorem~\ref{thm:sq-tree-intro}). Together, these results show that low-degree polynomials behave very differently in our setting than one might intuit based on previous work in related settings, such as in random constraint satisfaction problems or the block model.

\iffalse
A measure of complexity that attracted a lot of recent attention is the degree of polynomials required to achieve the task.
This is the complexity measure we use here. We show that for {\em all} chains with $\lambda_2 = 0$, ...

\todo{fill out the outline, maybe move some text here from later sections.}
Main motivations:
1. Understanding the power of low degree inference methods.
In particular, do they correspond to computational efficiency (need many references). There are known counter examples of algebraic nature (e.g. parities) but we claim it is much more than that and the broadcast tree process is a counter example. 

2. Previous work argued that data that is generated hierarichally requires complexity or depth to infer and that this is related to KS bound. We examine this question from the low-degree polys point of view and get similar answers/conjectures (need to check that low-degree complexity and circuit complexity results don't follow from one another? also need to say that our results hold for every chain with $\lambda_2 = 0$ while MMS only proved it for a specific chain) 

3. New insights about low-degree vs. robustness? 
We get a more nuanced picture of low degree vs. robustness.
Previous results (is it true??) argued that low-degree is the right way to consider robust statistica inference. 
In our model we can tolerate a small amount of noise but not a large amount of noise with BP, so it may really depend more on the amount of noise. 
\fi

%The Kesten-Stigum Bound is very important.
\subsection{Preliminaries}
\paragraph{Notation.} We use the standard notation $O_a(\cdot)$ to denote an upper bound with an implied constant which is allowed to depend on $a$; the notation $\poly_a(\cdot)$ is similar for denoting a bound which is polynomial in its parameters. 
We let $d_{TV}(P,Q)$ denote the total variation distance between distributions $P$ and $Q$ normalized to be in $[0,1]$ and we use $I(X; Y \mid Z)$ for the conditional mutual information of random variables $X,Y$ conditional on $Z$; see \cite{cover1999elements}. Given a vector $x$ and a subset of coordinates $S$, we let $x_S$ denote the $|S|$-dimensional vector 
%with coordinates of $x$ 
corresponding to elements of $S$. 
%We use the notation $\langle u, v \rangle = u^* v$ for the usual Hermitian inner product on $\mathbb{C}^n$. 

\paragraph{Markov processes on trees.}
%Let $\Sigma$ be a finite set which serves as the state space (alphabet) for the Markov process. Without loss of generality, we take 
We consider Markov processes with state space $\Sigma = [q] = \{1,\ldots,q\}$ where $q \ge 1$ is the size of the alphabet. 
Let $M : q \times q$ be the transition matrix of a time-homogeneous Markov chain on $\Sigma$, also referred to as the \emph{broadcast channel}. For simplicity, we always assume henceforth that $M$ is \emph{ergodic} (irreducible and aperiodic, see \cite{durrett2019probability}) so it has a unique stationary distribution $\pi_M$.  %\cite{mossel2004survey}.
Let $T = (V,E,\rho)$ be a rooted tree with vertex set $V$, root $\rho \in V$, and where $E$ is the set of directed edges $(u,v)$ where $u$ is the parent of $v$ in the corresponding tree. 
The \emph{broadcast process} on tree $T$ of depth $\ell$ with transition matrix $M$ and 
\iffalse root value $c$ is given by the probability measure
\[ \mu_{\ell,c}(x) := \bone(x_{\rho} = c) \prod_{(u,v) \in E} M_{x_u,x_v}. \]
More generally, the broadcast process with 
\fi 
root prior $\nu$ a probability measure on $[q]$ is given by
\[ \mu_{\ell,\nu}(x) := \nu(x_{\rho}) \prod_{(u,v) \in E} M_{x_u,x_v}. \]
%and when no probability measure $\nu$ is specified, it's assumed that the prior is $Uni([q])$. %unneeded?
%is the stationary distribution $\pi_M$ of the Markov chain with transition matrix $M$. 
\iffalse
When the root value $c$ is not fixed, we define the probability measure
\[ \mu_{\ell}(x) := (1/q) \prod_{(u,v) \in E} M_{x_u,x_v} \]
which corresponds to a mixture of the measures $\mu_{\ell,c}$ where the root value $c$ is sampled uniformly at random from $[q]$. (The choice of the uniform measure here is for simplicity, the results can be straightforwardly adapted to the case of an arbitrary prior on the root value.)
\fi
When not otherwise indicated, $\nu$ is the stationary distribution for $M$.
The probability measure $\mu$ is a \emph{Markov Random Field} on the tree $T$. %(where the orientation of the edges is ignored).
This means that if $A,B$ are subsets of the vertices of $T$ and all paths in $T$ from $A$ to $B$ pass through a third set of vertices $S$, then $X_A$ and $X_B$ are conditionally independent given $X_S$. This is called the \emph{Markov property}, see e.g. \cite{lauritzen1996graphical}.
%In the general model of Markov processes on trees, the transition matrix $M$ is allowed to differ between different edges of the tree and many results generalize to this setting in a natural way; however, we focus on the case where the chain is the same along every edge for simplicity.
In this paper, we focus on the setting of complete $d$-ary trees (i.e. trees where every non-leaf node has $d$ children, and all leaf nodes are at the same depth). For the $d$-ary tree of depth $\ell \ge 0$, we let $L$ be the set of leaves of the tree, i.e. the set of vertices in the tree at depth $\ell$. 
%\todo{more discussion}

\begin{defn}
We say that \emph{reconstruction is possible} on the $d$-ary tree with channel $M$ if 
\[ \inf_{\ell \ge 1} \max_{c,c' \in [q]} 
d_{TV} \Big(\mathcal{L}_{\mu_{\ell}}(X_L \mid X_{\rho} = c), \mathcal{L}_{\mu_{\ell}}(X_{L} \mid X_{\rho} = c') \Big)> 0 \]
where the notation $\mathcal{L}_{\mu}(X | E)$ denotes the conditional law of $X$ under $\mu$ given event $E$ occurs,
$\mu_{\ell}$ is the corresponding broadcast process on the depth $\ell$ tree with root $\rho$ and $L = L_{\ell}$ is the set of leaves. 
\end{defn}
When reconstruction is possible, the Bayes-optimal estimate of the root given the leaves can be computed in linear time by passing messages up the tree using the Belief Propagation algorithm \cite{mezard2009information}; for our purposes, we will not need the explicit formula for BP, which can be derived by applying Bayes rule, but refer the interested reader to the reference. 

Given a matrix $M$, we let $\lambda_2(M)$ denote the second-largest eigenvalue of $M$ in absolute value. 
The \emph{Kesten-Stigum (KS) threshold} on the $d$-ary tree is given by the equation $d |\lambda_2(M)|^2 = 1$. Building upon the original work of \cite{kesten1966additional}, it was shown that the KS threshold is sharp for the problem of \emph{count reconstruction} on trees \cite{mossel2003information}: count reconstruction is possible when $d |\lambda_2(M)|^2 > 1$ and impossible when $d |\lambda_2(M)|^2 < 1$. 
\begin{defn}
Let $C(x) := (\#\{i : x_i = c\})_{c \in [q]}$ be the function which computes count statistics of an input vector $x$ with entries in $[q]$. 
We say that \emph{count-reconstruction is possible} on the $d$-ary tree with channel $M$ if 
$\inf_{\ell \ge 1} \max_{c,c' \in [q]} d_{TV}(\mathcal{L}_{\mu_{\ell}}(C(X_L) \mid X_{\rho} = c), \mathcal{L}_{\mu_{\ell}}(C(X_{L})  \mid X_{\rho} = c')) > 0$
where the notation $\mathcal{L}(X | E)$ denotes the conditional law of $X$ given event $E$,
$\mu_{\ell}$ is the corresponding broadcast process on the depth $\ell$ tree and $L = L_{\ell}$ is the set of leaves on this tree. 
\end{defn}
Next, we define a notion of noisy reconstruction which plays an important role in this paper:
\begin{defn}\label{def:noisy-reconstruction}
For $\epsilon \in (0,1)$, we say that \emph{$\epsilon$-noisy reconstruction is possible} on the $d$-ary tree with channel $M$ if 
$\inf_{\ell \ge 1} \max_{c,c'} d_{TV}(\mathcal{L}_{\mu_{\ell}}(X'_L = \cdot \mid X_{\rho} = c), \mathcal{L}_{\mu_{\ell}}(X'_{L} = \cdot \mid X_{\rho} = c')) > 0$
where $X'$ is the $\epsilon$-noisy version of the broadcast process values $X$, generated by independently for each vertex $v$, setting $(X'_L)_v = (X_L)_v$ with probability $1 - \epsilon$ and otherwise sampling $(X'_L)_v$ from $Uni([q])$\footnote{More generally, our results hold where the noise
%process can sample from any 
is from any full support distribution on $[q]$.}.
%see discussion in \cite{janson2004robust}.} 
%probability $\epsilon$. %\todo{should we standardize using uniform vs stationary measure?}
\end{defn}
Note that in this definition, the law of $X'_L \mid X_{\rho} = c$ can equivalently be written as $\mathcal{L}_{\mu_{\ell}}(X_L \mid X_{\rho} = c) T_{\epsilon}$ where $T_{\epsilon}$ is the usual noise operator that independently resamples each coordinate of its input vector with probability $\epsilon$, see e.g. \cite{hopkins2018statistical,o2014analysis}. Finally, we recall from \cite{janson2004robust} the following standard definition: we say that \emph{robust reconstruction} on the $d$-ary tree with channel $M$ is possible if $\epsilon$-noisy reconstruction is possible \emph{for every} $\epsilon \in (0,1)$. 

\paragraph{Low-degree polynomials and computational-statistical gaps.} 
%% Discussed later
%The work \cite{brennan2020statistical} establishes connections between the well-established SQ (Statistical Query) framework \cite{kearns1998efficient} and the low-degree polynomial heuristic...
As discussed in the introduction, low degree polynomials have been studied in a wide variety of contexts and settings.
The recent work \cite{schramm2020computational} showed that a version of the low-degree polynomial heuristic can predict the recovery threshold for natural Bayesian estimation problems, even when the recovery threshold is below the detection/testing threshold. 
%such as the planted submatrix problem, even when the recovery threshold is below the detection/testing threshold. 
In the present work, we will use the following key definition from their paper\footnote{In our notation $X$ and $y$ are swapped compared to theirs, to match the convention in the broadcast process.}
%Notational remark: our $X$ is their $y$ and vice versa, as this is more consistent with the notation customarily used in the broadcast process.}:
\begin{defn}[Degree-$D$ Maximum Correlation \cite{schramm2020computational}]\label{def:deg-d-correlation}
Suppose that $(X,Y) \sim P$ where $X$ is a random vector in $\mathbb{R}^N$ and $Y$ is a random variable valued in $\mathbb{R}$. The \emph{degree-$D$ maximum correlation} is defined to be
\[ \text{Corr}_{ \le D}(P) := \sup_{f \in \mathbb{R}[X]_{\le D}, \E_P[f(X)^2] \ne 0} \frac{\E_P[f(X) \cdot Y]}{\sqrt{\E_P[f(X)^2]}} \]
where $\mathbb{R}[X]_{\le D}$ is the space of degree at most $D$ multivariate polynomials in variables $X_1,\ldots,X_N$ with real-valued coefficients. 
\end{defn}
\noindent
As explained there, when the target label $Y^*$ is a vector this definition can be applied with $Y$ equal to each of the coordinates of $Y^*$.
We note that we could rephrase our results in terms of a testing problem (as in much of the prior work on the low-degree method), but the above definition is  more natural in our context (it avoids the need to introduce a ``null distribution'' $Q$). In what follows, we omit the distribution $P$ the expectation is taken over as long as it is clear from context. 

%\paragraph{Degree of functions on $[q]^n$.}
Instead of referring to polynomial degree directly, we usually use the following more convenient and equivalent definition.
Suppose $f$ is a function $[q]^n \to \mathbb{R}$. We define the (Efron-Stein) \emph{degree} of $f$ to be the minimal $D$ such that there exist functions $f_S : [q]^{|S|} \to \mathbb{R}$ so that
$f(x) = \sum_{S \subset [n], |S| \le D} f_S(x_S)$.
%A natural way of defining the functions $f_S$ is given by looking at the
One such minimal choice of $f_S$ is the
Efron-Stein decomposition over $Uni [q]^n$, see e.g. \cite{o2014analysis}; this notion is also equivalent to the minimal degree polynomial representing $f$ where the variables are the one-hot encoding $x \mapsto (\bone(x_i = c))_{i \in [n], c \in [q]}$.
%and the notion of degree from
%the Efron-Stein decomposition matches the definition above, see e.g. \cite{o2014analysis}, and this notion is also equivalent to the minimal degree polynomial representing $f$ where the variables are the one-hot encoding $x \mapsto (\bone(x_i = c))_{i \in [n], c \in [q]}$.

\paragraph{Reconstruction below the KS threshold.} In this paper, we will largely consider the problem of tree reconstruction with matrices $M$ with $\lambda_2(M) = 0$; these exactly correspond to Markov chains which mix perfectly within a bounded number of steps. (There are many examples of such chains, for concreteness we give a very small example below in Example~\ref{example:chain}). Obviously, for such a chain $M$, $d |\lambda_2|^2 = 0$ for any value of $d$ so such a model is always below the Kesten-Stigum threshold. Nevertheless, based on general results from existing work we know that near-perfect reconstruction of the root is possible (e.g. using Belief Propagation, which computes the exact posterior distribution \cite{mezard2009information}). This is true as long as $d$ is sufficiently large, and even with a constant amount of noise $\epsilon$: %\todo{explain why we have this condition on $M$}
\begin{theorem}[\cite{mossel2003information}, Theorem~\ref{thm:noise-robust-reconstruction} below]\label{thm:noise-robust-reconstruction-intro}
Suppose $M$ is a the transition matrix of a Markov chain with pairwise distinct rows\footnote{This condition is needed to rule out the case of e.g. a rank one matrix $M$ where reconstruction is clearly impossible. See also \cite{mossel2003information} for a more complex and precise condition.}
%matter the value of $d$, as well as other closely related examples. A slightly more complicated version of this condition characterizes such matrices $M$ \cite{mossel2003information}.},
i.e. for all $i,j \in [q]$ the rows $M_i$ and $M_j$ are distinct vectors.
Let $\delta \in (0,1)$ be arbitrary. There exists $d_0 = d_0(M,\delta)$ and $\epsilon > 0$ such that for all $d \ge d_0$, $\epsilon$-noisy reconstruction is possible on the $d$-ary tree and furthermore there exists a polynomial-time computable function $f = f_{M,\ell}$ valued in $[q]$ such that
\[ \max_{c \in [q]} \Pr(f(X'_L) \ne X_{\rho} \mid X_{\rho} = c) < \delta \]
where $X'_L$ is the $\epsilon$-noisy version of $X_L$ (see Definition~\ref{def:noisy-reconstruction}).
%Then there exists $d_0 = d_0(M)$ such that for all $d \ge d_0$, there exists $\epsilon_d > 0$ such that $\epsilon_d$-noisy reconstruction is possible on the $d$-ary tree. \todo{define $\epsilon_d$-noisy reconstruction.} 
\end{theorem}
This exact statement does not appear in \cite{mossel2003information} but follows from arguments presented there; for completeness, we include a proof  (see Theorem~\ref{thm:noise-robust-reconstruction}). From the proof, we can see that a very simple recursive estimator is enough to solve this problem. 
\subsection{Our Results}
We study the power of low-degree polynomials for the problem of reconstructing the root of a Markov process. We consider this question in the context of two very closely related versions of the model which have both been extensively studied in the literature. 

%\begin{enumerate}
    %\item 
    \emph{Reconstruction with a known tree.} In this setting, the algorithm is given access to the leaf values from a single realization of the Markov process, and the goal is to estimate the root (where we are going to be interested in estimators which are low-degree polynomials of the leaves). The tree structure is known and the estimator/polynomial is allowed to depend on this information directly.
    %\item 
    
    \emph{Reconstruction with an unknown tree.} In this setting, the data is still generated by a complete $d$-ary tree but the tree (in other words, the true ordering of the leaves) is unknown to the algorithm. This version of the model has been extensively studied due to close connections to the problem of phylogenetic reconstruction in biology, see e.g. \cite{daskalakis2006optimal,Felsenstein:04,steel2016phylogeny}. Because this task is more difficult information-theoretically\footnote{Note that in the single-sample case ($m = 1$), the information available to the algorithm would only be count statistics, which we know are insufficient for reconstruction below the KS threshold \cite{mossel2004survey}.}, the algorithm is given access to $m$ i.i.d. samples from the broadcast model; we give a more precise definition of the model below. %\todo{annoying: what is the task the SQ algorithm wants to solve? reconstructing the tree it can do, but unclear what that means for low-degree polynomial algorithms. whereas low-degree can ask to reconstruct the root. should it be phrased as estimate the root of a new sample?}
%\end{enumerate}
%In the context of the present paper, %most of our results will be established first in the first model (the known tree setting) and then established for the second model by a reduction. The second model is of interest to us in part because it %has symmetry between the coordinates of the the input data, and thus 
%connects more directly to the previous literature on the low-degree polynomial heuristic as well as the SQ (Statistical Query) model --- see discussion below. 

\paragraph{Results for reconstruction with a known tree.}
We consider the problem of tree reconstruction with matrices $M$ with $\lambda_2(M) = 0$; these exactly correspond to Markov chains which mix perfectly within a bounded number of steps. As discussed above in Preliminaries, while such models are always below the Kesten-Stigum threshold for any value of the the arity $d$, under fairly weak conditions on $M$ the reconstruction problem is still solvable for $d$ sufficiently large (Theorem~\ref{thm:noise-robust-reconstruction-intro}). This is true even with noise and with a very simple reconstruction algorithm.
As our main result we show that despite the algorithmic tractability of this problem, only very high degree polynomials are able to get any correlation with the root, in the same sense as Definition~\ref{def:deg-d-correlation}. %Thus, low-degree polynomials and efficient algorithms behave very differently in this model.
\begin{theorem}[Corollary~\ref{corr:poly-fail} below]\label{thm:poly-fail-intro}
Let $M$ be the transition matrix of a Markov chain on $[q]$ and suppose that $1 \le k \le q$ is such that $M^k$ is a rank-one matrix.
For any function $f : [q]^L \to \mathbb{R}$ of Efron-Stein degree
at most $2^{\lfloor \ell/(k - 1) \rfloor}$ of the leaves $X_L$ and any prior $\nu$ on the root,
\[ \E[f(X_L) \cdot (\bone(X_{\rho} = c) - \nu(c))] = 0. \]
\end{theorem}
\begin{remark}[Tightness]\label{rmk:tightness}
For fixed $k$, this result is tight up to the base of the exponent. 
%because any function on the hypercube can be viewed as a multilinear polynomial and so 
When $M$ satisfies the assumption of Theorem~\ref{thm:noise-robust-reconstruction-intro}, there is a function on a constant-degree subtree to recover the root and (by Fourier expansion) this is a polynomial of degree $e^{O(\ell)}$.
\end{remark}
To interpret this result, observe that $N = d^{\ell}$, so taking $d$ a constant,
%(sufficiently large to satisfy the previous Theorem) 
the Theorem shows that polynomials of degree even $N^c$ for an explicit constant $c = c(d,k) > 0$ fail to get any correlation with the root label. In comparison, in the previously mentioned contexts in the low-degree polynomials literature, the threshold for polynomials of degree $O(\log N)$ matches the conjectured threshold for polynomial time algorithms (see e.g. \cite{hopkins2017efficient,hopkins2018statistical,kunisky2019notes}) and polynomials of degree $N^c$ correspond to conjectural thresholds for subexponential time algorithms (see e.g.  \cite{bandeira2019computational,ding2019subexponential}).

\paragraph{A consequence: subexponential sample complexity lower bound for the RBF  kernel.} As a consequence of our main result, we can analyze the behavior of kernel regression methods in our model. %We remind the reader that 
Kernel ridge regression is one of the canonical methods for solving supervised learning problems, including classification problems (see e.g. \cite{muthukumar2021classification}).
In many high-dimensional settings, it is believed that the function learned using a standard kernel (e.g. Gaussian or polynomial) is essentially a low-degree polynomial.
%well-modeled by the class of low-degree polynomials. 
%Obviously, this is the case for the polynomial kernel of degree $D$ which explicitly learns a degree $D$ polynomial (and is generally used with $D$ of moderate size). 
Standard results in learning theory  (\cite{shalev2014understanding}) imply that kernel ridge regression with a Gaussian/RBF (Radial Basis Function) kernel in $N$ dimensions can learn a degree $\ell$ polynomial on the hypercube or sphere using roughly $O(N^{\ell})$ samples. Establishing lower bounds on KRR is generally \emph{much harder}. In certain particularly tractable settings (e.g.\ data from the uniform distribution on the hypercube) it has been recently shown explicitly that kernel regression (only) learns a low-degree polynomial \cite{ghorbani2021linearized,mei2021generalization}. 

%Based on this context, one might guess 
It seems plausible to guess that kernel ridge regression with a standard kernel 
%like the RBF/Gaussian kernel (probably the most popular kernel in practice)
will require subexponentially many samples of (leaf label, root label) pairs in order to learn to predict the root. 
We are able to verify this prediction in the case of the popular RBF (Radial Basis Function) kernel. 
%--- note that the lower bound holds even if the hyperparameters (ridge penalty and bandwidth) are tuned optimally. 
See Section~\ref{sec:kernel} for formal notation and background on kernel ridge regression.
\begin{theorem}[Theorem~\ref{thm:kernel} below]\label{thm:kernel-intro}
Let $M$ be the transition matrix of a Markov chain on $[q]$ and suppose that $1 \le k \le q$ is such that $M^k = \pi \pi^T$ is a rank-one matrix, and suppose that $\pi$ has at least two nonzero entries.
%Suppose that $m/\delta \le e^{c N^{\epsilon}}$. 
Then the for any color $c \in [q]$ and prior $\nu$ for the root coloration $X_{\rho}$, the following is true. Given $m$ i.i.d. samples $(x_1,y_1),\ldots,(x_m,y_m)$ from the broadcast model on the $d$-ary tree with $N$ leaves and broadcast channel $M$, where $x_i$ is a one-hot encoded vector of leaf colorations and $y_i = 1(X_{\rho} = c) - \nu(c)$ is the centered indicator of the leaf coloration, we have that for any bandwidth $\sigma \ge 0$ and ridge parameter $\lambda \ge 0$, for $w$ the output of ridge regression in RKHS space with those parameters and feature map $\varphi$, that with probability at least $1 - \delta$
\[ \frac{\mathbb{E}_{x_0,y_0}[y_0 \langle w, \varphi(x_0) \rangle]}{\sqrt{\mathbb{E}_{x_0,y_0}[y_0^2]}} = O(\sqrt{1/N}) \]
provided that $m/\delta = O(e^{N^{a}})$ where $a = a(M,d) > 0$ is independent of the depth of the tree.
\end{theorem}
This establishes a new and illustrative example where KRR performs poorly in high dimensions, even though the ground truth is a relatively ``simple'' and the labels are closely related to the structure of the input data. 
Note that the conclusion %, which we stated in terms of a covariance, 
implies that $\mathbb{E}[(y_0 - \langle w, \varphi(x_0) \rangle)^2] \ge (1 - O(1/\sqrt{N}))\mathbb{E}[y_0^2]$, i..e. kernel ridge regression does not significantly outperform the constant zero estimator (``null risk'') unless it is given at least a subexponential number of samples. Also, as with Remark~\ref{rmk:tightness} this result is tight up to the power of the exponent $c$, since a subexponential degree polynomial exists which predicts the root well
%(and one which predicts the root well exists per the discussion in the Remark, under suitable assumptions on $M$)
and it can provably be learned with subexponential number of samples by KRR \cite{shalev2014understanding}.
%by setting the ridge parameter correctly due to standard generalization bounds (see \cite{shalev2014understanding}). 
%The proof of this result builds on the low-degree polynomials result; we explain in more detail in the Technical Overview. 

In Figure~\ref{fig:krr}, we test kernel ridge regression in a simulation in both the case $\lambda_2 = 0$ and $\lambda_2 > 0$: consistent with our result, KRR fails to beat the null risk when $\lambda_2 = 0$; interestingly, it also fails for moderately small values of $\lambda_2$ as well, which is related to the Open Problem we discuss later. In the figure,  KRR is performed using 2000 i.i.d. samples of $(x,y)$ pairs with $x$ the one-hot encoded leaf colorations and $y$ the centered indicator that the root color is $1$, as in Theorem~\ref{thm:kernel-intro}. Bandwidth and ridge penalty are selected via grid search on a validation set. The results for the baseline (RecMaj) are averaged over 16000 samples.
\begin{figure}
    \centering
    \includegraphics[scale=0.60]{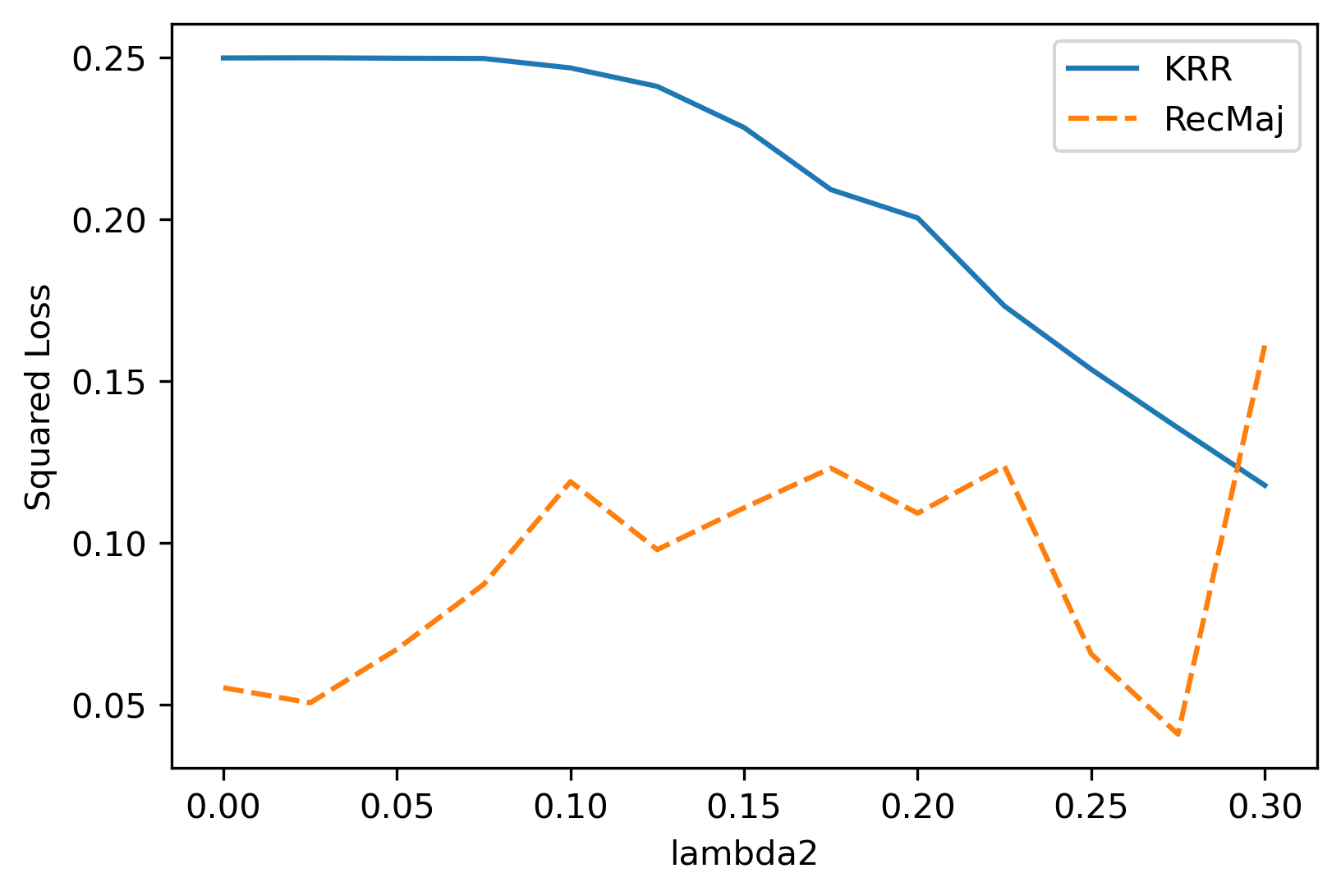}
    \caption{RBF Kernel Ridge Regression (KRR, blue line) test error in squared loss for predicting the root from leaves with data generated by  broadcast operator $M_{\lambda_2} := (1 - \lambda_2) M_0 + \lambda_2 I$ and varying $\lambda_2$. $M_0$ is from \eqref{eqn:M0}, and it can be directly checked that $\lambda_2$ is the second eigenvalue of $M_{\lambda_2}$. The tree is $10$-ary with depth $3$, and the prior is $\nu = (0.5,0.5,0)$. RecMaj (orange dotted line) is a baseline estimator which generalizes the one used in proof of Theorem~\ref{thm:noise-robust-reconstruction-intro}. Consistent with
     Theorem~\ref{thm:kernel-intro}, the output of KRR fails to correlate with the root coloration when $\lambda_2 = 0$ (since $0.25$ is the \emph{null risk}, the squared loss for the optimal constant predictor), even though %the Bayes error is much smaller than the null risk. In particular, 
     RecMaj correlates significantly with the root for all values of $\lambda_2$. In fact, KRR fails to correlate for all values of $\lambda_2$ up to around $0.1$, suggesting that the failure of KRR and low-degree polynomials should extend beyond $\lambda_2 = 0$. }
    \label{fig:krr}
\end{figure}
%First, we show that if the bandwidth parameter in the kernel is taken too small, then the output of Kernel Ridge Regression (KRR) is close to zero on a new test point. Otherwise, we show that any function with a substantial high degree polynomial component has large RKHS norm, and also construct an interpolator of the training data which has much smaller RKHS norm. It then follows that whatever the output of KRR is, it must have a small RKHS norm and cannot correlate with the true regression function.

\paragraph{Results for reconstruction with an unknown tree.}
Formally, we consider the following variant of the generative model which is  a variant of models in the phylogenetics literature. Briefly, in this model we generate %(leaf values of) 
$m$ i.i.d. realizations of the broadcasting model, where the tree is random and the prior on the root is biased towards a random root label. We include a parameter $\epsilon \ge 0$ which can be used to add noise to the final output the of model, just as above. 
\begin{defn}[$\epsilon$-Noisy Repeated Broadcast Model on Random Tree]\label{defn:repeated-broadcasting}
Let $\ell \ge 1, d \ge 2, m \ge 1, q \ge 1, \epsilon \ge 0$ and let $M$ be a Markov chain on $[q]$. Define $R = R_{\ell,d,m,M,\epsilon}$ by the following process:%\todo{simplify this by just having multiple iid samples from the tree process and let $\mathbb{Y}$ be the vector of root labels}
\begin{enumerate}
\item Sample $Y^* \sim Uni([q])$, and $\tau \sim Uni(S_N)$ is a random permutation. Let $T$ be the $d$-ary tree on the set of leaves ordered by $\tau$, so e.g. vertices $\tau(1)$ and $\tau(2)$ are siblings in $T$. 
\item Sample $X^{(1)},\ldots,X^{(m)}$ i.i.d. from the $\epsilon$-noisy broadcast process (see Definition~\ref{def:noisy-reconstruction}) on $T$ with prior $(2/3) \delta_{Y^*} + (1/3) Uni([q])$ and transition matrix $M$, where $\delta_{Y^*}$ is a delta distribution on $Y^*$. Let $\mathbb{X} = (X^{(1)}_L, \ldots, X^{(m)}_L)$. 
\end{enumerate}
\end{defn}
%Informally, t
The goal of the learning algorithm in the unknown tree model is this: given $m$ samples of the leaves of the broadcast process, encoded in $\mathbb{X}$,  reconstruct the root label $Y^*$ which the prior is biased towards\footnote{We could also consider the model where the root label is always $Y^*$. The soft bias we consider is nicer for minor technical reasons, and seems natural given we allow to add noise elsewhere in the model.}.  We discuss the reasons for defining the model this way:
%\begin{enumerate}
    %\item 
    1. The permutation $\tau$ ensures that the coordinates of $X^{(i)}_L$ behave in a symmetric way, or equivalently that the order of those coordinates is not semantically meaningful; observe that if we omitted it, then the first $d$ coordinates would always be neighbors in the tree. This is standard in the phylogenetics literature \cite{steel2016phylogeny} and this kind of symmetry is also assumed in the literature on low-degree polynomial hardness, see e.g. discussion in \cite{holmgren2020counterexamples}; in sparse PCA this is analogous to how the support of the planted sparse vector is chosen uniformly at random among size-$k$ subsets.
    %\item 
    2. The choice that root assignments are drawn from a tilted/biased distribution is different from the previous literature motivated by phylogenetics, where the root value is generally sampled fresh each time.
    %along with the rest of the broadcast process.
    This does not have a significant effect on how the algorithms used to estimate the tree work. The reason for our setup is to allow for straightforward comparison between SQ and low-degree polynomial models. If the root value was sampled from an unbiased measure each time, 
    %it would make sense to ask for a low-degree polynomial estimate of the root of a particular sample, but 
    it would not make sense for an SQ algorithm to estimate it, since SQ has no concept of individual samples. % to estimate this parameter, 
    %since the whole point of the SQ framework is to avoid any explicit reference to individual samples.
    %not particularly important. \todo{discussion. should we make the root value different each time and ask to estimate the root value of the first sample?}
%\end{enumerate}

To be formal, we define the Statistical Query VSTAT oracle analogue of $R$ in the usual way \cite{feldman2017statistical}. The oracle is defined conditional on $Y^*$ and the tree $T$, so the order of leaves in the tree will be consistent between different calls to the oracle. As a reminder, vector-valued queries are implemented in the SQ model by querying each coordinate of the vector individually.
\begin{defn}[$VSTAT(m)$ Oracle]
Let $Y^*$, $\tau$, $T$, and $M$ be as in Definition~\ref{defn:repeated-broadcasting}. Conditional on $Y^* = y^*$ and the tree $T = t$, we define $VSTAT(m)$ to be an arbitrary oracle which given a query function $\varphi : [q]^L \to [0,1]$, returns $p + \zeta_{\varphi}$ where $p := \E_R[\varphi(X^{(1)}) \mid Y^* = y^*, T = t]$ where $\zeta_{\varphi}$ is arbitrary (can be adversarily chosen) such that $|\zeta_{\varphi}| \le \max\left(\frac{1}{m}, \sqrt{\frac{p(1 - p)}{m}}\right)$.
\end{defn}

We now state our results in this model. Just as in the known tree case, there is a relatively simple algorithm which achieves nearly optimal performance in this setting when $\lambda_2(M) = 0$ and $d$ is a large constant, and furthermore this algorithm can straightforwardly be implemented in the SQ model described above. Establishing this  requires proving a new result in tree reconstruction, since (for example) the setting $\lambda_2(M) = 0$ which we care about rules out the use of Steel's evolutionary distance (see e.g. \cite{steel2016phylogeny,moitra2018algorithmic}) commonly used in reconstruction algorithms in phylogeny, as Steel's distance is only well-defined for nonsingular phylogenies, and some kinds of tree models with singular matrices are actually computationally hard to learn \cite{mossel2005learning}.
\begin{theorem}[Theorem~\ref{thm:sq-tree} below]\label{thm:sq-tree-intro}
Suppose $M$ is a the transition matrix of a Markov chain with pairwise distinct rows, i.e. for all $i,j \in [q]$ the rows $M_i$ and $M_j$ are distinct vectors, and suppose $\lambda_2(M) = 0$. There exists $d \ge 1$ and $\epsilon > 0$ so that the following result holds true for the complete $d$-ary tree with any depth $\ell \ge 1$. 
For any $\delta > 0$, there exist a polynomial time algorithm with sample complexity $m = poly_M(\log N,\log(1/\delta))$ from the $\epsilon$-noisy repeated broadcast model (Definition~\ref{defn:repeated-broadcasting})
%\todo{reference a precise description of the model from preliminaries, the color at the leaves is the same etc.}
which with probability at least $1 - \delta$:
%\begin{enumerate}
    %\item 
    1) outputs the true tree $T$ (equivalently, the true permutation $\tau$),
    %\item 
    2) outputs $\hat{Y}$ such that $\hat{Y} = Y^*$.
%\end{enumerate}
Also, this algorithm can be implemented %in the Statistical Query (SQ) model 
using a $VSTAT(m)$ oracle with $m = poly_M(\log(N/\delta))$ and polynomially many queries.
%and polynomial number of queries. 
\end{theorem}
% The lower bound for polynomials also applies here, just like in Theorem~\ref{thm:poly-fail-intro}. Since it is very similar, we leave the formal statement to the appendix (Theorem~\ref{thm:low-degree-failure2}).
Just as above, we show that the only polynomials which achieve any correlation at all with $Y^*$ have to be at least degree $N^c$:
\begin{theorem}[Theorem~\ref{thm:low-degree-failure2} below]\label{thm:low-degree-failure2-intro}
Let $M$ be the transition matrix of a Markov chain on $[q]$ and suppose that $1 \le k \le q$ is such that $M^k$ is a rank-one matrix.
%let $X_u \in \mathbb{F}_2^2$ be the random variable corresponding to the broadcast process on this tree with channel $M$. 
If $c \in [q]$ is arbitrary and $f$ is a polynomial with Efron-Stein degree strictly less than $2^{\lfloor \ell/(k - 1) \rfloor}$, then %$|S| < 2^{\ell}$, then
$\E_R[f(\mathbb{X})(\bone(Y^* = c) - 1/q)]  = 0$
where $R$ is as defined in Definition~\ref{defn:repeated-broadcasting}.
\end{theorem}
\subsection{Further Discussion}
\paragraph{Related work: complexity of reconstruction on trees.} 
\iffalse
Our work follows the work of Mossel and later Mossel who identified the Kesten-Stigum threshold as a complexity barrier in the context of the broadcast model on trees. 
In the broadcast mode on trees, two key parameters are the arity of the tree $d$ and the magnitude of the second eigenvalue $\lambda_2$ of the broadcast chain.  
It has been long known \cite{mossel2001reconstruction} %(second eigenvlau paper)
that for some tree broadcast models it is possible to estimate the root better than random even when $d \lambda_2^2 < 1$. We call the regime where reconstruction is possible but $d \lambda_2^2 < 1$ the hard regime. It is known that in the hard regime it is impossible to estimate the root better than random using the count of the different types (reference) and that no procedure for reconstruction is robust to sufficiently high noise on the leaves. 
\fi
Our work follows a line of previous work which identified the Kesten-Stigum threshold as a potential complexity barrier in the context of the broadcast model on trees. The work \cite{mossel2016deep} showed that algorithms that do not use correlation between different features (named ``shallow algorithms") cannot recover phylognies above the Kesten-Stigum threshold, where other (``deeper") algorithms can do so efficiently; 
the motivation in \cite{mossel2016deep} was to find simple data models where depth is needed for inference. 
More standard complexity measures were studied in \cite{moitra2020parallels} who obtained a number of results on the circuit complexity of inferring the root in the broadcast process. 
They conjectured that below the Kesten-Stigum threshold, inferring the root is $NC1$-complete and proved it %is $NC1$-complete 
for one specific chain satisfying $\lambda_2 = 0$. Although there are some connections between low-degree polynomials and certain circuit classes \cite{linial1993constant}, the results of this work and \cite{moitra2020parallels} are incomparable 
%(even for the particular chain they consider) 
and the techniques for establishing the lower bound are very different. 
Finally, we note the work \cite{jain2019accuracy} which studied the power of message-passing algorithms on finite alphabets: they proved such algorithms fail to recover all the way down to the Kesten-Stigum threshold, even in the simplest case of the binary symmetric channel with $q = 2$.

\paragraph{Message passing vs. low-degree polynomials.} One of the attractive properties of the class of low-degree polynomials is that it generally captures the power of a constant (or sufficiently slowly growing) number of iterations of message-passing algorithms such as BP, AMP, and Survey Propagation (see e.g. \cite{bresler2021algorithmic} and Appendix A of \cite{gamarnik2020low}), which is interesting since a constant number of steps of these algorithms are indeed useful for many statistical tasks. On the other hand, in our models, belief propagation (which computes the exact posterior) succeeds with high probability %(it is optimal, by definition)
whereas low-degree polynomials fail. This is not a contradiction:
%, but simply a consequence of the fact that in our setting, 
in our setting, BP requires $\Theta(\log N)$ iterations for the messages to pass from the leaves to the root and this is (as our main result shows) too large to simulate with low-degree polynomials. 

\paragraph{SQ and Low-Degree Polynomials.} The recent work \cite{brennan2020statistical} established sufficient conditions for predictions to match between the Statistical Query (SQ) and low-degree polynomial heuristic, in a general setting. Nevertheless, in the unknown tree setting we consider above we saw that SQ algorithms perform significantly better than low-degree polynomials. The results of \cite{brennan2020statistical} cannot be immediately applied to our setting, because we have phrased the problem as an estimation problem instead of a testing (a.k.a. distinguishing) problem; however, this is itself not the reason for the discrepancy as %our estimation problem could be rephrased
we could rephrase our problem in terms of testing the color of the root. 
Instead, the reason seems to be due to the ``niceness condition'' needed for their theory to apply.
%(see their work for the precise definition of this condition). Informally,
They show that the niceness condition will be satisfied for noise-robust problems when the ``null distribution'' in the testing problem is a product measure. %As we showed above in our positive results, this problem does satisfy the kind of noise robustness property they consider. 
Our setup is indeed noise robust (see Theorem~\ref{thm:noise-robust-reconstruction-intro}).
However, if we rephrased our problem as a testing one the null distribution will be a graphical model (with no bias at the root) and not a product measure.
%should be something like a broadcast model with no bias at the root, which is obviously not a product measure, since adjacent variables in the graphical model are noticeably correlated.
%However, in our setting there is no reasonable choice of a null distribution which will be a product measure, since adjacent variables in the graphical model are correlated and this is easy to test for. %(The absence of an obvious null distribution is not unique to our problem, but common with other 
%We note that the degree of polynomial needed to solve a Bayesian inference task was computed in \cite{koehler2018representational} for a sparse coding model (for which polynomial time algorithms exist, see e.g. the textbook \cite{moitra2018algorithmic}) --- in that case, unlike the present one, the degree needed was poly-logarithmic in the dimension which in the low-degree polynomials literature is viewed as consistent with the existence of polynomial time algorithms. 

Recently there has also been interest in understanding lower bounds against kernel learning algorithms (including polynomial kernels), in part motivated by connections to neural networks, and this involves connections to the SQ framework. See e.g. \cite{kamath2020approximate} and references within. These methods can, for example, prove strong lower bounds against learning parities with kernel ridge regression since parities have large SQ-dimension. See also \cite{koehler2018representational} for another example where polynomial degree lower bounds were established for a Bayesian inference task, though only polylogarithmic in the dimension.

\paragraph{Noise robustness and learning parities.} It was shown in \cite{janson2004robust} that the KS threshold is sharp for  \emph{robust reconstruction}, where we recall from above that (by definition) robust reconstruction on the $d$-ary tree with channel $M$ is possible if \emph{for all} noise levels $\epsilon \in (0,1)$, $\epsilon$-noisy reconstruction is possible. 
\iffalse
\[ \inf_{\ell \ge 1} \max_{c,c'} d_{TV}(\mathcal{L}_{\mu_{\ell}}(X'_L = \cdot \mid X_{\rho} = c), \mathcal{L}_{\mu_{\ell}}(X'_{L} = \cdot \mid X_{\rho} = c') > 0 \]
$\mathcal{L}(\cdot \mid \cdot)$ denotes the conditional law and
where $X'_L$ is generated from the leaf values $X_L$ on the $d$-ary tree of depth $\ell$ by independently for each leaf vertex $v$, setting $(X'_L)_v = (X_L)_v$ with probability $\epsilon$ and otherwise sampling $(X'_L)_v$ from $Uni([q])$ with probability $1 - \epsilon$.
\fi
%Then robust reconstruction on the $d$-ary tree with channel $M$ is possible when $d|\lambda_2(M)|^2 > 1$ and it is impossible when $d|\lambda_2(M)|^2 < 1$ --- in other words, robust reconstruction (in this sense) is determined by the second eigenvalue.
At first glance, this appears similar to the idea in the low-degree polynomials literature that the model should be \emph{slightly noisy} to rule out the example of learning parities (which can be solved in the noiseless setting by Gaussian elimination, but not when there is noise).
%but is conjecturally hard when there is any nontrivial amount of noise, see e.g. \cite{valiant2012finding}).
In fact the two notions are quite different: the robust reconstruction result shows that reconstruction becomes impossible for \emph{large noise levels} $\epsilon > \epsilon_*$ above a critical threshold $\epsilon_*$, whereas for small (but fixed) $\epsilon > 0$ the $\epsilon$-noisy reconstruction problem often \emph{remains solvable} --- see Theorem~\ref{thm:noise-robust-reconstruction-intro}. Our examples are fundamentally different to the parity example: (1) for the unknown tree version of our model, we showed that the problem is solvable with an SQ oracle whereas parities are well-known to be hard for SQ \cite{blum2003noise}, (2) relatedly, the algorithms which solve our problems are not ``algebraic'' in nature, and (3) our results hold irrespective of adding a small amount of noise, whereas learning parities with any constant amount of noise is conjecturally hard \cite{valiant2012finding}.
Altogether, we can think of these results as suggesting a new, more nuanced picture of low degree vs. robustness to noise. Whereas before the main dichotomy in the computational complexity of inference literature has been between zero noise and any noise, in our model reconstruction algorithms such as BP can tolerate a small amount of noise, but fail when the noise level crosses some critical threshold. At least in our setting with $|\lambda_2(M)| = 0$, low-degree polynomials appear to capture the latter ``large noise'' behavior instead of the ``small noise'' difficulty of the problem. 
\iffalse
One possible takeaway from this is that in general,  nuanced picture of low degree vs. robustness.
Previous results (is it true??) argued that low-degree is the right way to consider robust statistica inference. 
In our model we can tolerate a small amount of noise but not a large amount of noise with BP, so it may really depend more on the amount of noise. 
\fi

%reconstruction problem \emph{generally remains solvable with noise}. \todo{explicit result somewhere, and more discussion} 

\paragraph{Open Problem.}  What happens when $\lambda_2(M) \ne 0$?
%The following question asks 
It is natural to wonder if the Kesten-Stigum threshold $d|\lambda_2|^2 = 1$ is sharp for low-degree polynomial reconstruction, analogous to how it is sharp for robust reconstruction. Our main lower bound result (Theorem~\ref{thm:poly-fail-intro}) is consistent with this intuition.
%(failure of low-degree polynomials when $|\lambda_2(M)| = 0$)
%constitutes partial progress towards a positive answer. 
Also consistent with this intuition, our simulation result Figure~\ref{fig:krr} suggests that Kernel Ridge Regression may continue to fail for small but nonzero values of $\lambda_2(M)$. Formally, we ask:
\begin{question}[Kesten-Stigum is sharp for Low-Degree Polynomials?]
Suppose that $d$ and transition matrix $M$ are such that $d|\lambda_2(M)|^2 < 1$, i.e. we are below the Kesten-Stigum threshold. Let $D_{N}$ be an arbitrary function of $N = d^{\ell}$ such that $D_N = O(\log N)$ as $\ell \to \infty$. Is it true that the degree-$D_N$ maximum correlation between the broadcast process at the leaves $X_L$ and the root $X_{\rho}$ in the sense of Definition~\ref{def:deg-d-correlation} is asymptotically zero, i.e. $\lim\inf_{\ell \to \infty} \text{Corr}_{\le D_N} = 0$?
Equivalently, is it true that
\[ \lim\inf_{\ell \to \infty} \max_{c \in [q]} \sup_{deg(f) \le D_N, \E[f(X_L)^2] = 1} \E_{\mu_{\ell}}[f(X_L) (1(X_{\rho} = c) - \nu(c))] = 0 ? \]
%\todo{need to clarify that we are referring to $\mu$ and not $R$, fixme...}
\end{question}
Here we make the common choice of looking at $\log N$ degree polynomials (see e.g. \cite{hopkins2017efficient,kunisky2019notes}), but any degree is interesting. 

With the same intuition, we ask if a similar result to Theorem~\ref{thm:kernel-intro}, the lower bound for kernel ridge regression, holds below the Kestum-Stigum threshold --- see Figure~\ref{fig:krr} for related simulation results, which support the failure of KRR for small values of $\lambda_2$. We note that in our experiment the threshold where KRR starts to work is much closer to $d\lambda_2 = 1$. It is quite possible that this is a finite-depth effect since the experiment was done with a relatively shallow tree. 
Of course, if the sharp threshold is not the Kestum-Stigum threshold it would be extremely interesting to understand what the correct threshold is as a function of the broadcast model parameters. 
%See Section~\ref{apdx:open-problem} for a more formal statement and more discussion of this important question. %resolution of this question. It would be very interesting to understand the answer to this question for particular classes of models such as the Potts model on trees or assymetric binary channels.
% \begin{question}[Kesten-Stigum is sharp for Low-Degree Polynomials?]
% Suppose that $d$ and transition matrix $M$ are such that $d|\lambda_2(M)|^2 < 1$, i.e. we are below the Kesten-Stigum threshold. Let $D_{N}$ be an arbitrary function of $N = d^{\ell}$ such that $D_N = O(\log N)$ as $\ell \to \infty$. Is it true that the degree-$D_N$ maximum correlation between the broadcast process at the leaves $X_L$ and the root $X_{\rho}$ in the sense of Definition~\ref{def:deg-d-correlation} is asymptotically zero, i.e. $\lim\inf_{\ell \to \infty} \text{Corr}_{\le D_N} = 0$?
% Equivalently, is it true that
% \[ \lim\inf_{\ell \to \infty} \max_{c \in [q]} \sup_{deg(f) \le D_N, \E[f(X_L)^2] = 1} \E_{\mu_{\ell}}[f(X_L) (1(X_{\rho} = c) - \nu(c))] = 0 ? \]
% %\todo{need to clarify that we are referring to $\mu$ and not $R$, fixme...}
% \end{question}
% Here we make the common choice of looking at $\log N$ degree polynomials (see e.g. \cite{hopkins2017efficient,kunisky2019notes}), but any degree is interesting.
%the behavior of polynomials with other degrees is also interesting. % We pose additional open questions later in this work.  

\section{Technical Overview}
The detailed proofs of all results are given later. Here, we explain the high-level proof ideas, which we believe are relatively clean and conceptual. % and also give a few proof sketches.
%; this highlights the fact that our arguments are relatively clean and conceptual.
%\paragraph{A concrete example with $\lambda_2 = 0$.}
Before proceeding, %the reader may consider for the rest of the discussion
we give the following concrete example of a Markov chain $M_0$ with $\lambda_2(M_0) = 0$ and $q = 3$:
%. There are $q = 3$ states, and we have
\begin{equation}\label{eqn:M0}
M_0 = 
\begin{bmatrix}
0.5 & 0 & 0.5 \\
0.25 & 0.5 & 0.25 \\
0 & 1 & 0
\end{bmatrix}, \qquad M_0^2 = 
\begin{bmatrix}
0.25 & 0.5 & 0.25 \\
0.25 & 0.5 & 0.25 \\
0.25 & 0.5 & 0.25 
\end{bmatrix}.
\end{equation}
Since $M_0^2$ is rank one, it must be the case that $\lambda_2(M_0) = 0$.

\paragraph{Failure of low-degree polynomials (Theorem~\ref{thm:poly-fail-intro}).} We want to show that any low-degree polynomial $f$ of the leaves of the broadcast tree fails to correlate with the root. In general, it may be very difficult to compute the maximal correlation among all low-degree polynomials; what makes it possible in our case is that the correlation is exactly zero. If $c \in [q]$ is a color and $\nu$ is the prior at the root, we want to show
%\begin{equation}\label{eqn:no-correlation}
$\mathbb{E}[f(X_L)(1(X_{\rho} = c) - \nu(c))] = 0$. 
%\end{equation}
(Recall $L$ is the set of leaves and $X_L$ the leaf colorations.)
The first step is to use linear of expectation to break $f(X)$ into monomials: more formally, if $f(X) = \sum_{|S| \le D} f_S(X)$ is the Efron-Stein decomposition for a polynomial of degree $D$, then to show %\eqref{eqn:no-correlation}
the goal it clearly suffices to show
\[ \mathbb{E}[f_S(X_L)(1(X_{\rho} = c) - \nu(c))] = 0. \]
Crucially, the monomial $f_S$ is a function which depends only on a set of at most $D$ leaf colorations $X_D$. Therefore, the result follows if we can show those leaves by themselves are independent of the root coloration. This is shown by performing an \emph{iterative trimming} procedure on the minimal subtree spanned by the root and the leaves in $S$: every time there is an isolated path of length $k$ (where $M^k$ is rank one: $k = 2$ in the example above) all information is lost from the start of the path to its end. Using this idea and some elementary combinatorics, we can prove that if $|S|$ is small, the trimming procedure will delete everything, and so the root is indeed independent of these leaves. 
%To show this, we consider the minimal subtree spanned by the root and the leaves in $S$, and perform an \emph{iterative trimming} procedure: every time there is an isolated path of length $k$ ($k = 2$ in the example above) from node $u$ to its descendant $v$ in this subtree, where the tree does not branch along this path, $X_v$ has no mutual information with its ancestor $X_u$ because the Markov chain $M$ mixes perfectly in $k$ steps. By the Markov property, this implies that if we simply cut off the part of the tree from $v$ downward, we lose no information about the root or the remaining elements of $X_S$. Repeating this trimming process, we eventually get to either (1) an empty tree containing only the root (so we are done), or (2) a spanning subtree of the remaining elements of $X_S$ with no isolated paths of length $k$. By elementary combinatorics, in case (2) we have that if the broadcast tree has depth $\ell$ and $N$ leaves, the subtree from (2) must have roughly $2^{\ell/k} = N^{1/k\log(d)}$ leaves, and so the monomial must have at least that degree.

\paragraph{Failure of RBF Kernel Ridge Regression (Theorem~\ref{thm:kernel-intro}.)} This result builds on the low-degree polynomials result. First, we show that if the bandwidth parameter in the kernel is taken too small, then the output of Kernel Ridge Regression (KRR) is close to zero on a new test point and so it fails to learn anything. Otherwise, we can directly show that any function with a substantial high degree polynomial component has large RKHS norm.
%, which can be seen directly by looking at the explicit feature map for the Gaussian kernel when the bandwidth is large. 
We can also construct an interpolator of the training data which has much smaller RKHS norm, by showing that every training sample has a small ``fingerprint'' which uniquely identifies it and is detectable with a low-degree polynomial. It then follows that whatever the output of KRR is, it must have a small RKHS norm and cannot correlate with the true regression function.

\paragraph{Success of noise robust reconstruction using ``high degree'' algorithms \cite{mossel2003information}.} We briefly explain why noise robust reconstruction is possible with simple and computationally efficient algorithms. The key is to consider the case of a depth $1$ tree: because the rows of $M$ are distinct, if the degree of the tree is a sufficiently large constant, then by the Law of Large Numbers the empirical distribution of its children will be close to the row of $M$ corresponding to the state of the parent, letting us reconstruct the parent with say $99.9\%$ probability of success. Given this, it is not too hard to argue this argument works recursively and in the presence of a small adversarial noise.
%Since the rows of $M$ are far apart in $\ell_2$ norm, this can be made to work even if a small fraction of the children's values are changed adversarially (say $0.2 \%$). Applying this argument recursively, we can reconstruct a large fraction of the nodes at each level of the tree correctly and thus get the root correct with very high probability; furthermore, we see that this procedure is by construction robust to small amounts of noise.
Note that this algorithm recursively integrates \emph{global} information on the tree across multiple scales --- in contrast, the lower bound used the fact that low-degree polynomials can only aggregate information between  small sets of variables in a limited (linear) way. 

\paragraph{Unknown tree results (Theorem~\ref{thm:sq-tree-intro} and Theorem~\ref{thm:low-degree-failure2}).} The lower bound for low-degree polynomials in this setting can be reduced to the previous low-degree polynomial lower bound, which leaves proving that efficient (and SQ) algorithms can successfully solve this problem. Once we reconstruct the tree, we can run any algorithm for reconstructing the root given the leaves, e.g. the one described just above or BP. 
As far as reconstructing the tree, we first explain how to reconstruct the first layer. We prove that the joint distribution of any two leaves has enough information to tell us if they are immediate neighbors, which determines the location of all of their parents in the tree (such a test is easy to construct if we look at the \emph{generalized eigenvectors} of $M$). Now that the bottom layer of the tree structure is determined, we use the fact that we have very good estimates of their parents colorations using the algorithm described before. Crucially, since that algorithm's accuracy guarantee for reconstructing the internal node's colors is very strong and does not decay as we go further and further up the tree, we can indeed apply this argument recursively to get the whole tree. 

Implementing this algorithmic approach in the SQ framework is straightforward: at the end of the day it is based on computing the joint distributions of pairs of (estimated) vertex colorations, and those are all averages over the data. On the other hand, note that this method very strongly relies on the ability of an SQ algorithm to make \emph{adaptive queries}, since the queries made are based on the partially reconstructed tree structure, which is unknown to the algorithm before it starts.

%\printbibliography
%\bibliography{bib}

%\appendix
\section{Organization}
In the remaining sections we give full proofs of all results; there is no dependence on the Technical Overview as all the information there will be repeated here in more detail.  
%In Section~\ref{apdx:open-problem} we state more formally the open problem described in the intro. 
In Section~\ref{sec:known-tree} we prove the results for known trees: in particular, this includes the main lower bound result, which is the failure of low degree polynomials for recovering the root; we also show how to deduce the RBF kernel lower bound using this. This is also where the RecMaj algorithm from Figure~\ref{fig:krr} is formally explained.
In Section~\ref{sec:unknown-tree} we prove the results in the setting with an unknown tree; the main technical step is showing how to reconstruct the tree when $\lambda_2 = 0$ using a sample-efficient algorithm, which can be straightforwardly implemented in SQ. 
%\section{Open Problem: General Broadcast Chains}\label{apdx:open-problem}

\section{Known Tree: Upper and Lower Bounds}\label{sec:known-tree}
In this section, we prove a lower bound for arbitrary markov chains $M$ satisfying $\lambda_2(M) = 0$. From basic linear algebra (the existence of the Jordan Normal Form %equivalently the theory of generalized eigenvectors 
\cite{artin}), we know that $\lambda_2(M) = 0$ if and only if $M^k$ is a rank one matrix for some $1 \le k \le q$, i.e. the Markov chain mixes perfectly in a finite number of steps. For concreteness, we give an example of such a chain with $k = 2, q = 3$ below.
%% Fixed
%{\bf EM: Better to state it for general chains where $M$ is not rank $1$ and $M^2$ is. Maybe also state this can be done with a small constant amount of noise.See e.g. Theorem 6.1 in 
%\url{https://arxiv.org/pdf/math/0406446.pdf}}

%{\bf Example.} 
\begin{example}[Proof of Proposition 5, \cite{mossel2001reconstruction}]\label{example:chain}
The following Markov chain on $q = 3$ states is a simple example of a chain with $\lambda_2(M) = 0$: we have
\[
M = 
\begin{bmatrix}
0.5 & 0 & 0.5 \\
0.25 & 0.5 & 0.25 \\
0 & 1 & 0
\end{bmatrix}, \qquad M^2 = 
\begin{bmatrix}
0.25 & 0.5 & 0.25 \\
0.25 & 0.5 & 0.25 \\
0.25 & 0.5 & 0.25 
\end{bmatrix}.
\]
%and
%\[
%\]
%which is a rank-one matrix. 
\iffalse on $\mathbb{F}_2^2$ is a simple example of such a chain:
$(x,y) \mapsto (r, x+r)$ with probability $1/2$, and $(x,y) \mapsto (r,y+r)$ with probability $1/2$. The corresponding transition matrix is
\[ 
M = 
\begin{bmatrix}
0.5 & 0 & 0 & 0.5 \\
0.25 & 0.25 & 0.25 & 0.25 \\
0.25 & 0.25 & 0.25 & 0.25 \\
0 & 0.5 & 0.5 & 0
\end{bmatrix}
\]
where the first row corresponds to initial state $(0,0)$, the second to $(1,0)$, etc. Then
\[ M^2 = \begin{bmatrix} 1/4 & 1/4 & 1/4 & 1/4 \\ 1/4 & 1/4 & 1/4 & 1/4 \\
1/4 & 1/4 & 1/4 & 1/4 \\ 1/4 & 1/4 & 1/4 & 1/4 \end{bmatrix} \]
and so $\lambda_2(M) = 0$. 
\fi
\end{example}
\iffalse
\begin{theorem}[Theorem 2.1 of \cite{mossel2003information}]
Let $q \ge 2$ be arbitrary and let $M$ be the transition matrix of an arbitrary markov chain on $[q]$. Then there exists $d_0 = d_0(M)$ such that \todo{complete this}
\end{theorem}
\fi
\iffalse
\begin{theorem}[Proposition 5 of \cite{mossel2001reconstruction}]
If $d \ge 1000$, then reconstruction is possible on the $d$-ary tree with the above channel (even though $\lambda_2(M) = 0$).
\end{theorem}
\fi
\subsection{Failure of Low-Degree Polynomials}
\begin{theorem}\label{thm:no-MI}
Let $M$ be the transition matrix of a Markov chain on $[q]$ and suppose that $1 \le k \le q$ is such that $M^k$ is a rank-one matrix.
Let $S$ be any subset of the leaves of the depth-$\ell$ complete $d$-ary tree $T = (V,E,\rho)$ with root $\rho$ and let $(X_v)_{v \in V}$ denote the broadcast process on $T$ with channel $M$.
%let $X_u \in \mathbb{F}_2^2$ be the random variable corresponding to the broadcast process on this tree with channel $M$. 
Let $S$ be an arbitrary subset of the leaf nodes of this tree.
If $|S| < 2^{\lfloor \ell/(k - 1) \rfloor}$, then %$|S| < 2^{\ell}$, then
$I(X_{\rho};X_S) = 0$,
i.e. $X_S$ is independent of the root value $X_{\rho}$.
\end{theorem}
\begin{proof}
Assume for contradiction that  $|S| < 2^{\lfloor \ell/(k - 1) \rfloor}$ and $I(X_{\rho};X_S) > 0$. 
Let $T_S$ be the minimal spanning subtree of $T$ containing the root node $\rho$ and all of the elements of $S$. (Equivalently, $T_S$ is the union of all of the root-to-leaf paths to $S$.)

Recall that in our convention, the edges of the tree $T$ are directed from the parent to the child. 
We say that $T_S$ contains an \emph{isolated} length $k$ directed path if there exists adjacent nodes $u_0,\ldots,u_k$ contained in $T$ with $(u_i,u_{i + 1}) \in E$ for all $0 \le i < k$, and such that nodes $u_1,\ldots,u_{k - 1}$ all have degree $2$ in $T_S$.

We show that we can reduce to the case where $T_S$ contains no isolated length $k$ directed paths. Otherwise, let $u_0,\ldots,u_k$ be as defined above and let $S_{u_k}$ be the subset of $S$ consisting of descendants of $u_k$ (note that by the definition of $T_S$, $S_{u_k}$ is nonempty). Observe that
\[ I(X_{\rho};X_S) \le I(X_{\rho}; X_S,X_{u_k}) = I(X_{\rho}; X_{S \setminus S_{u_k}},X_{u_k}) \]
where the last equality follows by the Markov property (all nodes in $S_{u_k}$ are descendants of $u_k$, so $X_{S_{u_k}}$ is independent of the root value $X_{\rho}$ conditionally on $X_{S \setminus S_u}, X_{u_k}$).

Next, by the chain rule for mutual information
\[ I(X_{\rho}; X_{S \setminus S_{u_k}},X_{u_k}) = I(X_{\rho}; X_{S \setminus S_{u_k}}) + I(X_{\rho}; X_{u_k} \mid X_{S \setminus S_{u_k}}) = I((X_{\rho}; X_{S \setminus S_{u_k}}) \]
where the last equality follows from the fact that 
\[ I(X_{\rho}; X_{u_k} \mid X_{S \setminus S_{u_k}}) \le I(X_{u_0},X_{\rho}; X_{u_k}  \mid X_{S \setminus S_{u_k}}) = I(X_{u_0} ; X_{u_k}\mid X_{S \setminus S_{u_k}}) = 0  \]
where in turn the first equality follows from the Markov property ($X_{\rho}$ is independent of $X_{u_k}$ conditional on $X_{u_0}$ and $X_{S \setminus S_{u_k}}$) and the second equality follows because by the Markov property,
\[ X_{S \setminus S_{u_k}} \to X_{u_0} \to X_{u_k} \]
is a Markov chain where the rightmost channel has transition matrix $M^k$, a rank-one matrix, so the conditional law of $X_{u_k}$ is the stationary measure of $M$ regardless of the value of $X_{u_0}$, hence $X_{u_k}$ is conditionally independent of $X_{u_0}$. Combining the above claims shows that
\[ I(X_{\rho}; X_S) \le I(X_{\rho}; X_{S \setminus S_{u_k}}) \]
where $|S \setminus S_{u_k}| < |S|$; by monotonicity of mutual information we in fact have
\[  I(X_{\rho}; X_S) = I(X_{\rho}; X_{S \setminus S_{u_k}}). \]
Repeating this argument recursively reduces to the case where $T_S$ has no isolated length $k$ paths.

Finally, if $T_S$ has no isolated length $k$ paths then every internal node of $T_S$ is either: (a) at depth at most $k - 1$, or (b) has an ancestor at graph distance at most $k - 1$ away with degree at least $3$. By induction, this implies that the number of nodes at depth $\ell'$ in $T_S$ is at least twice as large as the number of nodes at depth $\ell' - (k - 1)$. Since $S$ is the set of nodes in $T_S$ at depth $\ell$, this implies that
\[ |S| \ge 2^{\lfloor \ell/(k - 1) \rfloor} \]
which completes our proof by contradiction. 
\end{proof}
\begin{corollary}\label{corr:poly-fail}
In the setting of the previous Theorem, for any function $f : [q]^L \to \mathbb{R}$ of Efron-Stein degree %Efron-Stein degree $D$ \todo{define efron-stein degree somewhere}
at most $2^{\lfloor \ell/(k - 1) \rfloor}$ of the leaves $X_L$ and any prior $\nu$ on the root,
\[ \E[f(X_L) \cdot (\bone(X_{\rho} = c) - \nu(c))] = 0. \]
\end{corollary}
\begin{proof}
By linearity of expectation and the Efron-Stein decomposition,
\[ \E[f(X_L) X_{\rho}] = \sum_{S \subset L, |S| \le D} \E[f_S(X_L) \cdot (\bone(X_{\rho} = c) - \nu(c))] \rangle] = 0 \]
where the last equality used the previous Theorem and the fact $\E[\bone(X_{\rho} = c)] = \nu(c)$. 
\end{proof}
\subsubsection{A consequence: failure of RBF kernel regression with oracle tuning}\label{sec:kernel}
\paragraph{Setting and notation.} We consider the performance of RBF kernel ridge regression (with arbitrary/oracle hyperparameter selection) for predicting the color of the root given the color of the leaves. As is customary, we encode the leaf vectors using a one-hot encoding, so the input to the regression is a list of i.i.d. samples $(x_i,y_i)_{i = 1}^m$ where $x_i$ is the vector of one-hot encoded leaves, i.e. $(x_i)_{\ell,c} = \bone(X_{\ell} = c)$, and for an arbitrary fixed color $c$, $y_i := \bone(X_{\rho} = c) - \nu(c)$ is the centered indicator that the root is colored $c$.
%\todo{make one-hot encoding into $-1,1$ to avoid having to compute norms.} %unneeded

\paragraph{Background on Kernel Ridge Regression.} We remind the reader of some standard facts about kernel ridge regression and the Gaussian/RBF kernel --- see \cite{shalev2014understanding} for a reference. Given a kernel $K(x,x')$, training points $x_1,\ldots,x_m$, and responses $y = (y_1,\ldots,y_m)$, the kernel ridge regressor with ridge parameter $\lambda$ is given by solving a linear equation 
\[ v = (\mathsf K + \lambda I)^{-1} y \]
where $\mathsf K_{ij} = K(x_i,x_j)$ is the kernel matrix, and the predicted response for a fresh data point $x_0$ is given by
\[ \hat y_0 := \sum_{i = 1}^n v_i K(x_i, x_0). \]
As is well-known, kernel ridge regression with ridge parameter $\lambda$ is equivalent to solving the  ridge regression problem 
\begin{equation}\label{eqn:ridge-regression}
\argmin_w \sum_{i = 1}^n (y_i - \langle w, \varphi(x_i))^2 + \lambda \|w\|_2^2 
\end{equation} with feature vectors $\varphi(x)$ lying in a certain Hilbert space. Note that with this parameterization the prediction for fresh data point $x_0$ would just be $\hat{y}_0 = \langle w, \varphi(x_0) \rangle$ since $w$ is in the Hilbert space.  In the case of the RBF kernel $K(x,y) = e^{-\|x - y\|^2_2/2\sigma^2}$, the corresponding feature map for $x \in \mathbb{R}^d$ is
\begin{equation}\label{eqn:rbf-map}
\varphi(x) = e^{-\|x\|^2/2\sigma^2} \left( \frac{1}{\sigma^{2(n_1 + \cdots + n_d)}} \frac{x_1^{n_1} \cdots x_d^{n_d}}{\sqrt{n_1! \cdots n_d!}} \right)_{n_1,\ldots,n_d \ge 0} 
\end{equation}
so that $K(x,y) = \langle \varphi(x), \varphi(y) \rangle$. Note that $\|\varphi(x)\| = 1$ since $K(x,x) = e^{0} = 1$.

\paragraph{Proof of the lower bound.}
We now proceed to prove the subexponential RBF sample complexity lower bound in our setting. 
\iffalse
First, based on the following Lemma we can show that if the bandwidth parameter $\sigma$ is taken to be too small, then the prediction on a new test point will be zero.
\begin{lemma}
Let $K(x,y)$ be an arbitrary real-valued function such that 
\end{lemma}
\begin{proof}

\end{proof}
\fi
For $\psi$ an element of the RKHS, define the orthogonal projection operator onto the space of degree $J$ and higher polynomials $P_{\ge J}$ by
\[ \left(P_{\ge J} \psi\right)_{n_1,\ldots,n_d} := \begin{cases}
0 & \text{if $n_1 + \cdots + n_d < J$} \\
\psi_{n_1,\ldots,n_d} & \text{otherwise}
\end{cases}.
\]
From the definition, we first show that for large degree $J$ and bandwidth $\sigma$ not too tiny, $P_{\ge J}$ is very contractive when operating on feature embeddings $\varphi(x)$.
\begin{lemma}\label{lem:high-degree-contraction}
For any $x \in \mathbb{R}^d$ and $\varphi(x)$ as defined in \eqref{eqn:rbf-map} with bandwidth parameter $\sigma > 0$,
\[ \|P_{\ge J} \varphi(x)\|^2 \le \frac{1}{\sqrt{J}} \left(\frac{e\|x\|^2}{J \sigma^2}\right)^J \]%\frac{\|x\|^{2J} e^J}{\sigma^{2J} \sqrt{J}J^J}. \]
\end{lemma}
\begin{proof}
First observe that
\[ \frac{(\|x\|^2/\sigma^2)^j}{j!} = \sum_{n_1 + \cdots + n_d = j} \frac{1}{n_1! \cdots n_d! (\sigma^2)^{2j}} x_1^{2n_1} \cdots x_d^{2n_d} \]
by applying the multinomial theorem. Therefore,
\[ \|P_{\ge J} \varphi(x)\|^2 = e^{-\|x\|^2/\sigma^2} \sum_{j = J}^{\infty} \frac{\|x\|^{2j}}{\sigma^{2j} j!} = e^{-\|x\|^2/\sigma^2} \frac{\|x\|^{2J}}{\sigma^{2J} J!} \sum_{j = 0}^{\infty}  \frac{\|x\|^{2j} J!}{\sigma^{2j} (J + j)!} \le \frac{\|x\|^{2J}}{\sigma^{2J} J!} \]
and then the stated result follows from a nonasymptotic version of Stirling's approximation. 
\end{proof}
Next, we prove that there exists a relatively low-degree and low-RKHS norm polynomial which perfectly interpolates the training data, by showing that with high probability every sample has a small and unique ``fingerprint'' given by looking at a small set of well-separated leaves.
\begin{lemma}\label{lem:fingerprint}
Let $M$ be the transition matrix of a markov chain on $[q]$ and suppose that $1 \le k \le q$ is such that $M^k = \pi \pi^T$ is a rank-one matrix, and suppose that $\pi$ has at least two nonzero entries. Then if $S$ is a set of leaves of distance at least $2k$ from each other and $X_1,\ldots,X_m$ are i.i.d. random vectors generated by the broadcast process with transition matrix $M$, the probability that there exists $i,j \in [m]$ such that $(X_i)_S = (X_j)_S$ is at most
${m \choose 2}\delta^{|S|}$
where $\delta = \delta(M) \in (0,1)$ is a constant depending only on $M$. 
\end{lemma}
\begin{proof}
First, let $X,X'$ be independent samples of the leaves from the generative model and let $c = c(M) > 0$ be such that the stationary distribution $\pi$ has at least two entries of size at least $c$. For $S$ a set of leaves of distance at least $2k$ from each other, we have by the Markov property that the entries of $X_S$ are independent from each other conditional on the values of the markov process $X_v$ for all vertices $v$ at height $k$ above the leaves; we see then that the conditional law of the leaves $X_S$ is $\pi^{\otimes S}$ which does not depend on $X_v$, so in fact the leaves $X_S$ are unconditionally distributed according to the product measure $\pi^{\otimes S}$. Then by independence,
\[ \Pr(X_S = X'_S) = \prod_{i \in S}\Pr(X_i = X'_i) \le (1 - c)^{|S|} \]
where in the last step we used that regardless of the value of $X_i$, $X'_i$ has a probability at least $c$ of being different from it. 
\end{proof}
\begin{lemma}\label{lem:indicator-to-rhs}
For $x \in \{0,1\}^d$ with $\sum_i x_i = p$, and $S \subseteq [d]$ and $b_S \in \{0,1\}^d$ arbitrary, there exists $w = w(p,S)$ of (Hilbert space) norm 
\[ \|w\|^2 \le 2^{|S|} e^{p/2\sigma^2} \max\left\{1,\sigma^{2|S|}\right\} \sqrt{|S|!}  \]
such that
\[ \langle w, \varphi(x) \rangle = 1(x_S = b_S). \]
\end{lemma}
\begin{proof}
Observe that
\[ 1(x_S = b_S) = \prod_{i \in S} [b_ix_i + (1 - b_i)(1 - x_i)] \]
which for fixed $b$, expands into a sum of at most $2^{|S|}$ many monomials of degree at most $|S|$ and with coefficient $1$. Representing this expanded polynomial in the RKHS, using \eqref{eqn:rbf-map}, then leads to the stated norm bound. 
\end{proof}
We show that the overlap between two independent samples of the leaves from the model concentrates exponentially with a subgaussian tail:
\begin{lemma}\label{lem:overlap-concentration}
Let $M$ be the transition matrix of a markov chain on $[q]$ and suppose that $1 \le k \le q$ is such that $M^k = \pi \pi^T$ is a rank-one matrix. Then if $X_L,X'_L$ are two independent random vectors of leaf colorations generated by the broadcast process on the $d$-ary tree with $N = |L|$ leaves and $x_L,x'_L$ are the corresponding one-hot encodings, we have that
\[ \Pr\left(\left|\frac{1}{N} \langle x_L, x'_L \rangle - \|\pi\|_2^2\right| >  t\right) \le 2e^{-cNt^2} \]
where $c = c(M,d) > 0$ is a constant not depending on $N$.
\end{lemma}
\begin{proof}
First, observe that if $N$ is smaller than $d^k$, this bound can be proved trivially by shrinking $c$, so henceforth we assume $N$ is larger than this. By the law of total probability, it is sufficient to prove the desired bound conditional on the colors $X_V,X'_V$ where $V$ is the set of vertices at height $k$ above the leaves, and similar to the proof of Lemma~\ref{lem:fingerprint} we observe by the Markov property that this makes the color of the set of children of any particular $v \in V$ independent of the colors of all non-children of $v$. This means that $\langle x_L, x'_L \rangle$ a sum of bounded independent random variables, and because $M^k = \pi \pi^T$ we have that its expectation is $\|\pi\|_2^2 N$, so the result follows immediately from Hoeffding's inequality \cite{vershynin2018high}.
\end{proof}
\begin{theorem}\label{thm:kernel}
Let $M$ be the transition matrix of a markov chain on $[q]$ and suppose that $1 \le k \le q$ is such that $M^k = \pi \pi^T$ is a rank-one matrix, and suppose that $\pi$ has at least two nonzero entries.
Suppose that $m/\delta \le e^{c N^{\epsilon}}$. Then given $m$ i.i.d. samples $(x_1,y_1),\ldots,(x_m,y_m)$ from the broadcast model on the $d$-ary tree with $N$ leaves and broadcast channel $M$, we have that for any bandwidth $\sigma \ge 0$ and ridge parameter $\lambda \ge 0$, for $w$ the output of ridge regression in RKHS space with those parameters, that with probability at least $1 - \delta$
\[ \frac{\mathbb{E}_{x_0,y_0}[y_0 \langle w, \varphi(x_0) \rangle]}{\sqrt{\mathbb{E}_{x_0,y_0}[y_0^2]}} = O(\sqrt{1/N}) \]
provided that $m/\delta = O(e^{N^{\epsilon}})$ where $\epsilon = \epsilon(M,d) > 0$ is independent of $N$ (equivalently, independent of the depth of the tree). 
\end{theorem}
\begin{proof}
As usual, we will use that $N$ can be assumed larger than a fixed absolute constant without loss of generality.
The proof is via case analysis on the bandwidth parameter $\sigma$. 

First we make an argument which covers the case of small bandwidth parameter $\sigma$. Note that for any $i$, $\|x_i\|^2 = N$ almost surely since there are $N$ leaves and each leaf is one-hot encoded. 
By Lemma~\ref{lem:overlap-concentration} and the union bound, with probability at least $1 - \delta/4$ for any $i \ne j$ in $[m]$ we have
\[ \|x_i - x_j\|_2^2 = 2N - 2\langle x_i, x_j \rangle \ge 2(1 - \|\pi\|_2^2)N - O_{M,d}(\sqrt{N\log(m/\delta)}) \]
so 
\[ \mathsf K_{ij} = e^{-\|x_i - x_j\|_2^2/2\sigma^2} \le \exp\left([-(1 - \|\pi\|_2^2)N + O_{M,d}(\sqrt{N\log(m/\delta)})]/\sigma^2\right). \]
It follows that there exists %$c_1 = c_1(M,d) > 0$ and 
$c_2 = c_2(M,d) > 0$ such that %if $m/\delta \le e^{c_1 N}$ and
if $\sigma \le c_2N^{1/2 - \epsilon/4}$, then $(\mathsf K)_{ij} \le e^{-N^{\epsilon/3}}$ for $i \ne j$ and so by Gershgorin's disk theorem and the fact that the diagonal of $\mathsf K$ is all-ones, $\|\mathsf K - I\|_{OP} \le 1/N$. Hence for the Kernel Ridge solution $v = (\mathsf K + \lambda I)^{-1} y$ we have $\|v\| \le 2\|y\| \le 2\sqrt{m}$. 

Consider a fresh test set of independently sampled pairs of leaf and root colorations $(x'_1,y'_1),\ldots,(x'_{ms},y'_{ms})$ where $s := N\log(2/\delta)$. Observe by Hoeffding's inequality that with probability at least $1 - \delta/4$,
\[ \left|\frac{1}{ms} \sum_{i = 1}^{ms} \left(\sum_{j = 1}^m v_j K(x_j,x_0)\right)^2 - \mathbb{E}_{x_0}\left[\left(\sum_{j = 1}^m v_j K(x_j,x_0)\right)^2\right] \right| = O(\sqrt{\log(2/\delta)/s})  \]
where $x_0$ is a fresh one-hot encoded vector of leaf colorations sampled from the same distribution and
where we used the fact that $\|v\| \le 2\sqrt{m}$ and $K(\cdot,\cdot) \le 1$ to show that over the randomness of $x_0$, $\left|\sum_{j = 1}^m v_j K(x_j,x_0)\right| \le 2\sqrt{m}$ almost surely, which we used in order to apply Hoeffding's inequality.
By repeating the argument used to show the off-diagonal entries of $\mathsf K$ are small, we have with probability at least $1 - \delta/4$
\[ \frac{1}{ms} \sum_{i = 1}^{ms} \left(\sum_{j = 1}^m v_i K(x_j,x_i)\right)^2 \le m e^{-N^{\epsilon/2}}, \]
%hence by Hoeffding's inequality for a fresh data point $x_0$ 
hence by the triangle inequality we have with probability at least $1 - \delta$ that \[ \E_{x_0}\left[\left(\sum_{i = 1}^m v_i K(x_i,x_0)\right)^2\right] \le me^{-N^{\epsilon/2}} + O(\sqrt{\log(2/\delta)/s}) \] 
and recalling $s = N\log(2/\delta)$ gives the result in this case. 
%\fnote{fill in details}

Now we cover the remaining set of bandwidth parameters where $\sigma > c_2 N^{1/2 - \epsilon/4}$. By the combination of Lemma~\ref{lem:fingerprint} applied with $|S| = C_{M} \log (m/\delta)$ and Lemma~\ref{lem:indicator-to-rhs}, we have that there exists $w$ such that for every $x_i$
\[ \langle w, \varphi(x_i) \rangle = y_i \]
and
\begin{equation}\label{eqn:norm-bound}
\|w\| \le (m/\delta)^{C'_M} e^{N/4\sigma^2} \sigma^{C_M \log m/\delta} \sqrt{(C_M \log m/\delta)!}. 
\end{equation}
It follows that the output of KRR with any ridge parameter $\lambda \ge 0$ has norm at most the rhs of \eqref{eqn:norm-bound} (otherwise, replacing the output with $w$ would shrink the norm without decreasing the training error in \eqref{eqn:ridge-regression}). Next, by Lemma~\ref{lem:high-degree-contraction} we have that for any $x$ and degree $J$
\begin{align*}
\langle P_{\ge J} w, \varphi(x) \rangle 
&= \langle w, P_{\ge J} \varphi(x) \rangle \\
&\le \|w\| \|P_{\ge J} \varphi(x)\| \\
&\le \|w\| \frac{1}{\sqrt{J}} \left(\frac{e\|x\|^2}{J \sigma^2}\right)^J 
=\|w\| \frac{1}{\sqrt{J}} \left(\frac{e N}{J \sigma^{2}}\right)^J
%\le \|w\| \frac{1}{\sqrt{J}} \left(\frac{e N^{\epsilon/2}}{J c_2}\right)^J 
\end{align*}
so taking as in Corollary~\ref{corr:poly-fail} $J = 2^{\lfloor \ell/(k - 1) \rfloor} = N^{\epsilon}$ where this equation defines $\epsilon$ and using that 
\[ N/J\sigma^2 = N^{1-\epsilon}/\sigma^2 = O(\sigma^{-\epsilon/(1 - \epsilon/2)}), \]
we have that
for any $w$ satisfying \eqref{eqn:norm-bound}, 
\begin{align*} 
|\langle P_{\ge J} w, \varphi(x) \rangle|
&\le \|w\| (c_3/\sigma^{\epsilon/(1 - \epsilon/2)})^{N^{\epsilon}} \\
&\le (m/\delta)^{C'_M} e^{N^{\epsilon/2}/4c_2} \sigma^{C_M \log m/\delta} \sqrt{(C_M \log m/\delta)!} (c_3/\sigma^{\epsilon/(1 - \epsilon/2)})^{N^{\epsilon}} = O((1/\sigma)^{N^{\epsilon}/2}).
\end{align*}
Since by Corollary~\ref{corr:poly-fail} and Cauchy-Schwarz we have that
\[ \mathbb{E}_{x_0,y_0}[y_0 \langle w, \varphi(x_0) \rangle] = \mathbb{E}_{x_0,y_0}[y_0 \langle P_{\ge J} w, \varphi(x_0) \rangle] \le \sqrt{\mathbb{E}_{x_0,y_0}[y_0^2]} \sqrt{\mathbb{E}_{x_0,y_0}[ \langle P_{\ge J} w, \varphi(x_0) \rangle^2]}\]
combining this with the bound on $|\langle P_{\ge J} w, \varphi(x) \rangle|$ completes the proof.
\iffalse
Then for a fresh (i.e. independent) data point $x_0,y_0$ we have by the $L_2$ triangle inequality and Cauchy-Schwarz that
\begin{align*} 
\sqrt{\mathbb{E}_{x_0,y_0}[(y_0 - \sum_{i = 1}^m v_i K(x_i,x_0))^2]} &\ge \sqrt{\mathbb{E}_{x_0,y_0}[y_0^2]} -  \sqrt{\E_{x_0}[(\sum_{i = 1}^m v_i K(x_i,x_0))^2]} \\
&\ge \sqrt{\mathbb{E}_{x_0,y_0}[y_0^2]} - 2\sqrt{m}\sqrt{\sum_{i = 1}^m \E_{x_0}[K(x_i,x_0)^2]} 
\end{align*}
\fi
\end{proof}

\subsection{Success of noise-robust reconstruction using non-low-degree algorithms}
Above we saw that when $|\lambda_2(M)| = 0$, very high degree polynomials are needed to get any estimate correlated with the root. Nevertheless, for ``most'' matrices $M$ with $|\lambda_2(M)| = 0$ and for degree $d$ sufficiently large as a function $M$ there exists a simple recursive and noise-robust method which witnesses the fact that reconstructing the root is possible. If one likes, this recursive function can trivially be expressed as a polynomial: then it will be a very high-degree polynomial that is nonetheless robust to noise. 

The reason for the qualifier ``most'' in the discussion above is that there are some degenerate $M$ for which the task is clearly impossible: e.g. if $M$ is rank one (so it does not depend on its input). There are other similar examples, e.g. the chain on 3 states which deterministically transitions from state $1$ to state $2$, and such that at states $2$ and $3$ the chain flips a fair coin to transition to either state $2$ or $3$. With this clarified, we can now state the known positive result for reconstruction. 

\begin{theorem}[Theorem 6.1 of \cite{mossel2004survey}]
Suppose $M$ is a the transition matrix of a Markov chain with pairwise distinct rows, i.e. for all $i,j \in [q]$ the rows $M_i$ and $M_j$ are distinct vectors. Then there exists $d_0 = d_0(M)$ such that for all $d \ge d_0$, reconstruction is possible on the $d$-ary tree. 
\end{theorem}
A variant of the condition in this Theorem gives a tight characterization of Markov chains where reconstruction is possible on the infinite $d$-ary tree for sufficiently large $d$, see Theorem 2.1 of \cite{mossel2003information}. 

By revisiting the proof of Theorem, we get the following slightly more precise result which we will use in later sections. This result shows that for any desired accuracy $\delta$, for sufficiently large degrees $d$ there exists a noise-tolerant estimator $f$ which reconstructs the root correctly with probability at least $1 - \delta$ uniformly of the color of the root.
\begin{theorem}[Proof of Theorem 2.1 of \cite{mossel2003information}]\label{thm:noise-robust-reconstruction}
Suppose $M$ is a the transition matrix of a Markov chain with pairwise distinct rows, i.e. for all $i,j \in [q]$ the rows $M_i$ and $M_j$ are distinct vectors.
Let $\delta \in (0,1)$ be arbitrary. There exists $d_0 = d_0(M,\delta),\epsilon > 0$ such that for all $d \ge d_0$, $\epsilon$-noisy reconstruction is possible on the $d$-ary tree and furthermore there exists a polynomial-time computable function $f = f_{M,\ell}$ valued in $[q]$ such that
\[ \max_{c \in [q]} \Pr(f(X'_L) \ne X_{\rho} \mid X_{\rho} = c) < \delta \]
where $X'_L$ is the $\epsilon$-noisy version of $X_L$ (see Definition~\ref{def:noisy-reconstruction}).
%Then there exists $d_0 = d_0(M)$ such that for all $d \ge d_0$, there exists $\epsilon_d > 0$ such that $\epsilon_d$-noisy reconstruction is possible on the $d$-ary tree. \todo{define $\epsilon_d$-noisy reconstruction.} 
\end{theorem}
\begin{proof}[Proof sketch.]
As explained above, this result follows from examination of the proof of Theorem 2.1 in \cite{mossel2003information}. For the reader's convenience, we summarize the main idea of the proof. 

In the base case of a depth $1$ tree, reconstruction of the root with probability at least $1 - \delta$ is possible provided $d$ is a suitably large constant, because by basic large deviations theory (Sanov's Theorem \cite{large-deviations}) the empirical distribution of the children will concentrate around the row of $M$ corresponding to the root label (which by assumption is distinct from all of the other rows). This procedure is also robust to a small amount of noise, which handles the case where $\epsilon > 0$ and in fact even if the $\epsilon$ proportion of children assigned labels by the noise process choose their labels adversarially. When doing the induction, the result of the reconstruction process at lower levels of the tree can therefore (by conditional independence) be modeled as the true values with a small amount of adversarial noise and this allows the same argument to show that at each level each vertex is recovered correctly with probability at least $1 - \epsilon$ (where we take $\epsilon := \delta$).
%In the presence of noise ($\epsilon_d > 0$), the same happens except the distribution of the children concentrates around the appropriate convex combination of the row of $M$ with the uniform measure, which still means the row of $M$ is recoverable. 
%\todo{add a proof sketch.}
\end{proof}
\begin{remark}[RecMaj in Figure~\ref{fig:krr}]
The RecMaj algorithm in Figure~\ref{fig:krr} corresponds to the algorithm described in the above proof sketch: i.e. a recursive algorithm which to reconstruct the coloration of a vertex, looks at the reconstructions of its children, takes the empirical distribution, and picks the corresponding row of $M$ which is closest in $\ell_2$ norm. 
\end{remark}
\iffalse
The following Theorem from \cite{mossel2003information} gives a tight characterization of matrices where reconstruction is possibly with sufficiently large degree:
\begin{defn}[Defn 2.1 of \cite{mossel2003information}]
Let $M : q \times q$ be the transition matrix of a Markov chain on $[q]$ and let $\sim_M$ be the minimal equivalence relation on $[q]$ satisfying that 
\[ \left(\forall c \in [q], \sum_{c' : c \sim c'} M_{a,c'} = \sum_{c' : c \sim c'} M_{b,c'}\right) \longrightarrow a \sim b. \]
We call $\sim_M$ the \emph{indistinguishability relation} induced by $M$. 
\end{defn}
\begin{theorem}[Theorem 2.1 of \cite{mossel2003information}]

\end{theorem}
\fi
%TODO: prove a noise robust version of Theorem 2.1 of \cite{mossel2003information} using the same technique.
\subsection{Low-Degree Polynomials succeed above the KS threshold}
The Kesten-Stigum threshold is the sharp threshold for \emph{count reconstruction} defined earlier. The definition of count reconstruction informally says that there is a nontrivial amount of mutual information between count statistics at the leaves and the value of the Markov Random Field at the root.   
To relate count reconstruction to low-degree polynomials, we use the following more precise result: %, which follows by specializing the proof of the count reconstruction result on the $d$-ary tree:
%More precisely, we have the following result:
\begin{lemma}[Proof of Theorem 1.4 of \cite{mossel2003information}]
Suppose that $d |\lambda_2(M)|^2 > 1$. %Then there exists a random variable $S$ with
There exist coefficients $s_c \in \mathbb{C}$ for $c \in [q]$ such that the random variable
\[ S = \sum_{c \in [q]} s_c \#\{X_{\ell} = c : \ell \in L\} \]
satisfies
\[ \E[S \mid X_{\rho} = c] = v_c \]
where $v$ is an unit-norm eigenvector of $M$ in its second-largest eigenspace, i.e. achieving $\|M v\| = |\lambda_2(M)|$, and such that
\[ \E[|S|^2 \mid X_{\rho} = c] \in [A,B] \]
where $0 < A \le B$ are constants depending only on $d$ and $M$ (in particular, they are independent of the depth of the tree). 
\end{lemma}
As a consequence of this, we immediately obtain that low-degree polynomials (in fact, degree 1 polynomials) have nontrivial correlation with the root above the KS threshold, in the same sense as Definition~\ref{def:deg-d-correlation}.
\iffalse
\begin{theorem}
If $d|\lambda_2(M)^2| > 1$ then 
\[ \frac{\E_P[f(X) \cdot Y]}{\sqrt{\E_P[f(X)^2]}} \]
\end{theorem}
\fi
%Janson and Mossel \cite{janson2004robust} showed that the Kesten-Stigum bound is the sharp threshold for \emph{robust reconstruction}. The upper bound part of this statement follows from the fact that \emph{count reconstruction} is possible. 
\subsubsection{A Question: Bayes-Optimal Reconstruction}
We saw above that degree-1 polynomials of the leaves are sufficient to achieve nontrivial correlation with the root, provided that the model we consider is above the KS threshold. A natural question is whether higher degree polynomials have a significant advantage over degree-1 polynomials for estimating the value of the root. Relevant to this question, we recall the following result and conjecture from \cite{mossel2014belief} which concerns noise-robust recovery with the Binary Symmetric Channel (equivalently, the Ising model on trees without external field):
\begin{theorem}[Theorem 3.2 of \cite{mossel2014belief}]\label{thm:mns}
There exists an absolute constant $C \ge 1$ such that the following result is true. 
For $\theta \ge 0$ let
\[ M = \begin{bmatrix}
(1 + \theta)/2 & (1 - \theta)/2 \\
(1 - \theta)/2 & (1 + \theta)/2
\end{bmatrix} \]
and observe that $\lambda_2(M) = \theta$.
If $d \theta^2 > C$, then for all $\epsilon < 1$ and $X'_L$ defined by the $\epsilon$-noisy broadcast model,
\begin{align*} 
&\lim_{\ell \to \infty} d_{TV}(\mathcal{L}_{\mu_{\ell}}(X'_L = \cdot \mid X_{\rho} = 1), \mathcal{L}_{\mu_{\ell}}(X'_{L} = \cdot \mid X_{\rho} = 0)) \\
&= \lim_{\ell \to \infty} d_{TV}(\mathcal{L}_{\mu_{\ell}}(X_L = \cdot \mid X_{\rho} = 1), \mathcal{L}_{\mu_{\ell}}(X_{L} = \cdot \mid X_{\rho} = 0))
\end{align*}
in other words, if $\epsilon < 1$ is fixed then in the limit of infinite depth the probability of reconstructing the root correctly is the same as in the noiseless case $\epsilon = 0$.
\end{theorem}
(Recall that the equivalence of the statement in terms of TV and in terms of maximum probability of reconstructing the root follows from the Neyman-Pearson Lemma \cite{neyman1933ix}.)
This statement is conjectured to hold with $C = 1$ \cite{mossel2014belief} and as explained there, is closely related to Bayes-optimal recovery in the stochastic block model. Based on this, we ask the following question:
\begin{question}
Do polynomials of degree $O(\log N)$ achieve asymptotically Bayes-optimal recovery with the above channel when $d \theta^2 > 1$? More precisely, does there exist a polynomial threshold function $f$ of degree $O(\log N)$ which asymptotically achieves
\[ \Pr(f(X_L) = X_{\rho}) = (1 + o(1)) \Pr(\text{sgn}(\E[X_{\rho} \mid X_L] - 1/2) = X_{\rho}) \]
where the rhs is the error of the Bayes-optimal estimator. 
\end{question}
It seems likely the answer to this question is positive. The reason for this is the following: (1) if the conjectured strengthening of Theorem~\ref{thm:mns} is true, then it implies that the combination of a majority vote up to some depth and $\omega(1)$ number of rounds of belief propagation achieves Bayes-optimal recovery, and (2) a constant or very slowly growing number of rounds of belief propagation can be simulated with low-degree polynomials (see Appendix of \cite{gamarnik2020low}), and the threshold used in the majority vote should also be approximable by polynomials. We state the conjecture with $O(\log N)$ degree polynomials since this is informally considered to correspond to ``polynomial time algorithms'' in the low-degree framework \cite{hopkins2018statistical,kunisky2019notes}, but based on the above discussion it seems likely that a smaller degree than $O(\log N)$ is sufficient, e.g. any degree going to infinity with $N$ may be sufficient.
%\todo{unclear if theorem/conjecture implies bayes-optimal reconstruction with low degree polynomials due to issues implementing threshold operations with rare events. investigate}

%Show that Bayes-Optimal reconstruction is possible under BSC conjecture, but fails for general channels (example of two independent chains put together).

\section{Unknown Tree Setting}\label{sec:unknown-tree}
In this section, we show that for any channel $M$ satisfying the conditions of Theorem~\ref{thm:noise-robust-reconstruction}, i.e. such that for sufficiently large $d$ reconstructing the root is possible (in the known tree setting/in the usual sense), then in the unknown tree setting that a relatively simple algorithm succeeds at reconstructing the root with a polynomial number of samples, and this algorithm can be straightforwardly implemented in the SQ (Statistical Query) model with polynomial number of queries and error tolerance.
%with polynomial number of queries and polynomial inverse tolerance. 

The key step in the algorithm for reconstructing the root is a method of reconstructing the tree, which lets us reduce to the known tree setting. This kind of problem has previously been extensively studied in the context of phylogenetic reconstruction with particular channels $M$ coming from biology, and for example algorithms with polynomial runtime and sample complexity are known in the case that $M$ is a nonsingular matrix \cite{mossel2005learning}. In the present context, we are very interested in the case of singular matrices (e.g. those with $\lambda_2(M) = 0$) so we cannot rely on existing results. 

\paragraph{Model.} We remind the reader that in the unknown tree setting, we are in the model of Definition~\ref{defn:repeated-broadcasting}. This means that an unknown $Y^*$ is sampled from $Uni([q])$, and the algorithm seeks to reconstruct $Y^*$ given access to $m$ i.i.d. samples $X^{(1)}_L,\ldots, X^{(m)}_L$ of the leaves generated by the broadcasting process with root prior $(2/3) \delta_{Y^*} + (1/3) Uni([q])$, i.e. the root is biased/tilted towards the unknown $Y^*$. When we say the tree is ``unknown'' in this model, it means that the algorithm is not given a priori knowledge of the true order of the leaves, e.g. the algorithm does not know at the beginning whether coordinates $1$ and $2$ of $X^{(1)}_L$ correspond to siblings or to leaves far apart in the tree (this is completely analogous to the situation in phylogenetic reconstruction \cite{steel2016phylogeny}). In the definition of this model, this is modeled by shuffling the order of the leaves by an unknown permutation $\tau$; note that this order is kept consistent between each sample.
\subsection{Failure of low-degree polynomials}
\begin{theorem}\label{thm:low-degree-failure2}
Let $M$ be the transition matrix of a Markov chain on $[q]$ and suppose that $1 \le k \le q$ is such that $M^k$ is a rank-one matrix.
%let $X_u \in \mathbb{F}_2^2$ be the random variable corresponding to the broadcast process on this tree with channel $M$. 
If $c \in [q]$ is arbitrary and $f$ is a polynomial with Efron-Stein degree strictly less than $2^{\lfloor \ell/(k - 1) \rfloor}$, then %$|S| < 2^{\ell}$, then
\[ \E_R[f(\mathbb{X})(\bone(Y^* = c) - 1/q)]  = 0 \]
where $R$ is as defined in Definition~\ref{defn:repeated-broadcasting}.
\end{theorem}
\begin{proof}
Let $\nu(c) = 1/q$ for $c \in [q]$ denote the prior on $Y^*$.

By linearity of expectation and the definition of Efron-Stein degree, it suffices to show the result for functions $f$ of the form $f_{S_1}(X^{(1)}_L) \cdots f_{S_m}(X^{(m)}_L)$ where $\sum_i |S_i| < 2^{\lfloor \ell/(k - 1) \rfloor}$, where each $f_{S_i}(X^{(i)}_L)$ is a function only of the coordinates of its input in $S_i$. Since the samples $X^{(1)},\ldots,X^{(m)}$ are conditionally independent given the value of $Y^*$, we have 
\begin{align*}
    &\hspace{-1cm}\E_R\left[\left(\prod_{i = 1}^m f_{S_i}(X^{(i)}_L)\right) (\bone(Y^* = c) - \nu(c))\right] \\
    &= \E_R\left[\E\left[\left(\prod_{i = 1}^m f_{S_i}(X^{(i)}_L)\right) (\bone(Y^* = c) - \nu(c)) \mid Y^* \right]\right] \\
    &= \E_R\left[\left(\prod_{i = 1}^m \E[f_{S_i}(X^{(i)}_L) \mid Y^*]\right) (\bone(Y^* = c) - \nu(c))\right] \\
    &= \E_R\left[\left(\prod_{i = 1}^m \E[f_{S_i}(X^{(i)}_L)]\right) (\bone(Y^* = c) - \nu(c))\right]
    = 0
\end{align*}
where in the first equality we used the law of total expectation, in the second equality we used the aforementioned conditional independence, in the third equality we crucially used that by Theorem~\ref{thm:no-MI} the low-degree polynomial $f_{S_i}(X^{(i)}_L)$ is independent of the root value and thus $Y^*$, and in the last step we used that $Y^* \sim \nu$ by definition. 
\end{proof}
\subsection{Reconstruction Algorithm}
For $c \in [q]$, let $e(c)$ or $e_c$ denote the $q$th standard basis vector in $\mathbb{R}^q$. In both cases, the vector is a column vector.
\begin{lemma}\label{lem:marginal-lb}
Suppose that $\nu$ is a probability measure on $[q]$ and $\nu(c) > 0$ for all $c \in [q]$, then there exists a constant $\alpha = \alpha(M,\nu) > 0$ such that the following is true. Let $(X_u)_u \sim \mu$ for $u \in V$ be defined by the broadcasting process on $T = (V,E,\rho)$ with prior $\nu$ at the root and channels corresponding to $M : q \times q$ the transition matrix of an ergodic Markov chain. Then $\mu(X_u = c) > \alpha$ for all $u \in V$.
\end{lemma}
\begin{proof}
Under the assumptions, there exists some $\beta > 0$ such that $\nu = \beta \pi_M + (1 - \beta) \nu'$ for $\nu'$ a probability measure. Because $\pi_M$ is the stationary distribution and the marginal law at any vertex $u$ is $\nu M^k$ for some $k \ge 0$, it follows that $\mu(X_u = c) > \beta \pi_M(c) \ge \min_c \beta \pi_M(c) =: \alpha > 0$.
\end{proof}
\begin{lemma}\label{lem:moment-relatives}
Suppose that $u,v$ are two descendants of node $w$ at graph distance $k$ from $w$
and random variables $X_u,X_v,X_w$ follow
 the Markov process on trees $\mu$ with transition matrix $M : q \times q$. Then
\[ \E[e(X_u) e(X_v)^T] = (M^k)^T \Pi_w M^k \]
where $\Pi_w : q \times q$ is a diagonal matrix with entries the marginal law of $X_w$, i.e. $(\Pi_w)_{cc} = \mu(X_w = c)$ for $c \in [q]$. 
\end{lemma}
\begin{proof}
Using the law of total expectation and using by the Markov property that $X_u$ and $X_v$ are conditionally independent given $X_w$, we have
\[ \E[e(X_u) e(X_v)^T] = \E[ \E[e(X_u) \mid X_w] \E[e(X_v)^T \mid X_w]] = \E[ (e_{X_w}^T M^k)^T (e_{X_w}^T M^k)] = (M^k)^T \Pi_w M^k   \]
where in the last equality we used the definition of $\Pi_w$ and the definition of the broadcast process in terms of the transition matrix $M$.
\end{proof}
Based on this, we can recursively reconstruct the tree when the degree is sufficiently large. We note that for other channels like the BSC channel, tree reconstruction methods often handle internal nodes $u$ by computing majorities of the nodes under them, which gives an unbiased estimate of the spin $X_u$, but this technique is not applicable in our setting (it's unclear that unbiased estimators exist). Nevertheless, we show that applying the estimator from Theorem~\ref{thm:noise-robust-reconstruction} can be used in a similar way, provided the degree $d$ is sufficiently large.
\begin{theorem}\label{thm:sq-tree}
Suppose $M$ is a the transition matrix of a Markov chain with pairwise distinct rows, i.e. for all $i,j \in [q]$ the rows $M_i$ and $M_j$ are distinct vectors. If $|\lambda_2(M)| > 0$, additionally suppose that the prior on the root of the tree is the stationary distribution of $M$.
There exists $d \ge 1$ and $\epsilon > 0$ so that the following result holds true for the complete $d$-ary tree with any depth $\ell \ge 1$. 
For any $\delta > 0$, there exist a polynomial time algorithm with sample complexity $m = poly_M(\log N,\log(1/\delta))$ from the $\epsilon$-noisy repeated broadcast model (Definition~\ref{defn:repeated-broadcasting})
%\todo{reference a precise description of the model from preliminaries, the color at the leaves is the same etc.}
which with probability at least $1 - \delta$:
\begin{enumerate}
    \item outputs the true tree $T$ (equivalently, the true permutation $\tau$)
    \item outputs $\hat{Y}$ such that $\hat{Y} = Y^*$.
\end{enumerate}
Also, this algorithm can be implemented in the Statistical Query (SQ) model using a $VSTAT(m)$ oracle with $m = poly_M(\log(N/\delta))$ and polynomial number of queries. 
\end{theorem}
\begin{proof}
Given that the algorithm can correctly output the true tree $T$, the fact that it outputs the correct root label follows straightforwardly from Theorem~\ref{thm:noise-robust-reconstruction} by using the algorithm specified in that result to estimate the root in each sample, and then taking the majority vote over those samples (which will succeed with high probability provided we take $\Omega(\log(2/\delta))$ samples due to Hoeffding's inequality), and this can approach can also clearly be implemented in the SQ model (the SQ query is the robust reconstruction function of the leaves which outputs a vector, so we take the expectation of this and look at the largest entry of this vector). In the remainder of the proof, we show how to correctly output the true tree $T$ with high probability.

We first prove the result in the case that $\lambda_2(M) \ne 0$ and afterwards describe how to modify the argument straightforwardly when $\lambda_2(M) = 0$. Let $\varphi$ be a right eigenvector such that $M \varphi = \lambda_2 \varphi$. 
We start by describing the algorithm which computes the estimated tree $\hat{T}$ from the bottom up: let $\alpha = \alpha(M,\delta) > 0$ be a parameter to be set later. Let $\hat \E[\cdot]$ denote the expectation over the empirical distribution of $m$ samples, so for any function $f$ we have $\hat E[f(X)] = \frac{1}{m} \sum_i f(X^{(i)})$. %\todo{explain what this notation means.}
\begin{enumerate}
    \item Base case: for all leaves $u \ne v$ define $g(u,v) := |\langle \varphi, \hat E[e(X_u) e(X_v)^T] \varphi \rangle|$. Let $g_{max} = \max_{u \ne v} g(u,v)$ and set $u,v$ to be neighbors in $\hat{T}$ iff $g(u,v) \ge g_{max} - \alpha$. This constructs the first layer of the tree $\hat{T}$.
%    \todo{technically $\Pi$ is not known. can we compute traces or something instead.}
    \item Recursive case: suppose that we have reconstructed the first $s \ge 1$ layers of the tree (from the bottom), and the current layer of the tree has more than one element. For each pair of internal nodes $u,v$ at the current level of the tree, let $S_u,S_v$ be the set of leaves under these nodes and let $g_s(u,v) := |\langle \varphi, \hat E[e(f_{M, \ell - s}(X_{S_u})) e(f_{M, \ell -s}(X_{S_v}))^T] \varphi \rangle|$ where $f_{M,\ell - s}$ is as defined in Theorem~\ref{thm:noise-robust-reconstruction}.
    Let $g_{max} = \max_{u \ne v} g(u,v)$ and set $u,v$ to be neighbors in $\hat{T}$ iff $g(u,v) \ge g_{max} - \alpha$. This constructs the next layer of the tree $\hat{T}$.
\end{enumerate}
We now need to show that with total probability at least $1 - \delta$, $\hat{T} = T$.
First we consider the behavior of the base case; for simplicity, we first describe the argument when $\epsilon = 0$. Observe that if $u$ and $v$ are siblings in $T$ at depth $\ell$ then by Lemma~\ref{lem:moment-relatives}
\[ \langle \varphi, \E[e(X_u) e(X_v)^T] \varphi \rangle = |\lambda_2|^2 \langle \varphi, \Pi_{\ell -1} \varphi \rangle \]
where $\Pi_{\ell - 1}$ is a diagonal matrix encoding the marginal law of $X_w$ for any $w$ at depth $\ell - 1$,
and similarly, if $u$ and $v$ are not siblings then they are at graph distance at least $4$ in $T$ so
\[ \langle \varphi, \E[e(X_u) e(X_v)^T] \varphi \rangle \le |\lambda_2|^4 \langle \varphi, \Pi_{\ell - 1} \varphi \rangle \]
which is smaller by a factor of $|\lambda_2|^2$. (Note, here we are using the fact that in the case $|\lambda_2| > 0$, we additionally assumed the prior at the root is stationary and so the marginal law  at every depth in the tree is the stationary distribution.)
Observe that by Hoeffding's inequality and the union bound we have that with probability at least $1 - \delta/n$ that in the base case step, every entry of the matrix $\hat{\E}[e(X_u) e(X_v)^T]$ for every pair of leaves $u \ne v$ is within additive error $O(\sqrt{\log(n/\delta)/m})$ of its expectation. It follows from this and Lemma~\ref{lem:marginal-lb} that if $\alpha = (1/C_M) (|\lambda_2|^2 - |\lambda_2|^4)$
for $C_M$ a sufficiently large constant depending only on $M$, %\todo{minor fixme: needs to be scaled by a notion of $\langle \varphi, \Pi \varphi \rangle$. Would be simplified if we make the assignment of the root a soft constraint instead of hard.}
 $\epsilon$ is sufficiently small with respect to $\alpha$, 
and $m = \Omega_M(\log(n/\delta))$ then in the base case the algorithm computes neighbors correctly. Observe that at each layer, if the algorithm has correctly reconstructed $T$ in all previous layers then the sets $S_u$ for all nodes $u$ in this layer are deterministic functions of $T$, and hence so are the queries the algorithm makes to $\hat E$. By a similar application of the union bound and Hoeffding's inequality as well as Theorem~\ref{thm:noise-robust-reconstruction} and the assumption that $d$ is sufficiently large with respect to $M$ it follows that the algorithm succeeds at all subsequent layers as well. 

Note that provided we take $\epsilon > 0$ is sufficiently small, we can show the base case of the argument will still succeed by using the triangle inequality, and the inductive step in the argument will succeed because of Theorem~\ref{thm:noise-robust-reconstruction}. 
%Provided that $m = \Omega_M(\log(n/\delta))$ and $\alpha$ is chosen appropriately \todo{fixme} this ensures the algorithm correctly reconstructs the

Finally, in the case that $\lambda_2(M) = 0$, we let $\varphi$ be a generalized eigenvector such that $M \varphi \ne 0$ but $M^2 \varphi = 0$. Note that such a vector must exist because, $0$ is an eigenvalue of algebraic multiplicity $q - 1$ as $M$ is ergodic and $\lambda_2 = 0$, and because our assumption on $M$ rules out the case that $M$ is rank one, so it's Jordan normal form must have at least one Jordan block with size at least $2$ and this corresponds to the existence of such a generalized eigenvector $\varphi$. Now observe for such a $\varphi$ that if $u,v$ are siblings in $T$ at depth $\ell$ then
\[ \langle \varphi, \E[e(X_u) e(X_v)^T] \varphi \rangle = \langle M \varphi, \Pi_{\ell -1} M \varphi \rangle \]
which by Lemma~\ref{lem:marginal-lb} is lower bounded by a constant $C'_M > 0$,
while if $u,v$ are not siblings,
\[ \langle \varphi, \E[e(X_u) e(X_v)^T] \varphi \rangle = 0. \]
Setting $\alpha = C'_M/2$ and defining the remaining constants similarly to above ensures the algorithm succeeds, by the same argument. 

Note that in both the case $\lambda_2(M) \ne 0$ and $\lambda_2(M) = 0$, the algorithm is implemented by taking the expectation of certain functions over the samples, so it is straightforwardly implementable with SQ queries by replacing the empirical expectation with the VSTAT oracle. 
%a very similar argument works where we let $\varphi$ be a generalized eigenvector such that $M \varphi \ne 0$ but $M^2 \varphi = 0$. \todo{elaborate}
\end{proof}

%\paragraph{Acknowledgements.}  
\iffalse
\begin{corollary}
In the same setting as the previous Theorem~\ref{thm:sq-tree} and with the same runtime and sample complexity requirements, there exists $C = C(M,d) > 0$ not depending on the depth of the tree and an algorithm which outputs $\hat{Y}$ such that $\Pr(\hat Y = Y^*) \ge 1 - \delta$, provided that reconstruction is possible on the $d$-ary tree with channel $M$. Also, this algorithm can be implemented in the SQ model in the same sense as in Theorem~\ref{thm:sq-tree}.
\end{corollary}
\begin{proof}
Using the tree reconstruction algorithm from Theorem~\ref{thm:sq-tree}, we can assume knowledge of the true tree $T$ with high probability. By one of the conclusions of Theorem~\ref{thm:sq-tree}, we know that reconstruction is possible on the $d$-ary tree with channel $M$ (in the usual sense, i.e. for arbitrary depth). %From the definition, this means that there exists a pair of colors $c,c'$ such that their ($\epsilon$-noisy) leaf distributions have positive total variation distance, which 
This implies (see Proposition 2.1 of \cite{mossel2004survey}) that the value of the process at the leaves has positive mutual information with the value at the root \todo{how to conclude there is a $\hat{Y}$ from this, silly question.}
%, which in turn implies that the average posterior distribution of the root of 
\end{proof}
\fi
%\bibliographystyle{plain}
%\bibliography{bib}
\appendix
\iffalse
\section{Reconstruction in the SQ Model below the KS Threshold}
In this Appendix, we verify that the algorithms for reconstructing the root which were previously developed in the literature can be implemented in the SQ model. \todo{complete this sketch}
\fi

\bibliographystyle{plain}
\bibliography{bib}

\end{document}